\documentclass[onefignum,onetabnum]{siamart171218}
\usepackage{graphicx}
\usepackage{amsmath}
\usepackage{amscd}
\usepackage{lmodern}
\usepackage{enumitem}
\usepackage[latin2]{inputenc}
\usepackage{t1enc}
\usepackage[mathscr]{eucal}
\usepackage{indentfirst}
\usepackage{graphicx}
\usepackage{graphics}
\usepackage{pict2e}
\usepackage{amssymb}
\usepackage{wrapfig}
\usepackage{epic}
\numberwithin{equation}{section}
\usepackage[margin=2.9cm]{geometry}
\usepackage{epstopdf} 
\usepackage{bibcheck}
\usepackage{caption}
\usepackage{mdframed}
\usepackage[english]{babel}
\usepackage{graphicx}
\usepackage{listings}
\usepackage{mathtools}
\usepackage[font=small,labelfont=bf]{caption}
\usepackage{xcolor}
\usepackage{bm}
\usepackage{amsfonts}
\usepackage{amsmath}
\usepackage{multirow}
\usepackage{algpseudocode}
\usepackage[colorinlistoftodos]{todonotes}
\newcommand\myeq{\stackrel{\mathclap{\normalfont\tiny\mbox{def}}}{=}}

\newcommand{\R}{\mathbb{R}}
\newcommand{\sm}{\setminus}

\newcommand{\supp}{\mathrm{supp}}

\newcommand{\argmin}{\textnormal{argmin}}
\newcommand{\argmax}{\textnormal{argmax}}

\newcommand{\tx}{\textnormal}

\newcommand{\Sc}[2]{(#1, #2)}

\newcommand{\dist}{\textnormal{dist}}

\newcommand{\s}[1]{\{1, ..., #1\}}
\newcommand{\n}[1]{\| #1 \|}

\usepackage{rotating}
\newtheorem{Th}{Theorem}[section]
\newtheorem{Lemma}{Lemma}[section]
\newtheorem{Cor}{Corollary}[section]
\newtheorem{Rem}{Remark}[section]
\newtheorem{Prop}{Proposition}[section]
\newtheorem{Def}{Definition}[section]

\title{Active set complexity of the Away--step Frank--Wolfe Algorithm}

\author{ Immanuel~M.~Bomze\thanks{ISOR, VCOR \& ds:UniVie,
 	Universit\"{a}t Wien, Austria
 	(\tt{immanuel.bomze@univie.ac.at})}
  \and
  Francesco~Rinaldi\thanks{Dipartimento di Matematica ``Tullio Levi-Civita'', Universit\`a
    di Padova, Italy
     (\tt{rinaldi@math.unipd.it})}
     \and 
     	Damiano~Zeffiro\thanks{Dipartimento di Matematica ``Tullio Levi-Civita'', Universit\`a
		di Padova, Italy
		(\tt{damiano.zeffiro@math.unipd.it})}
}

\begin{document}

\maketitle

 \begin{abstract} In this paper, we study active set identification results for the away-step Frank-Wolfe algorithm in different settings. We first prove a local
 identification 
 property that we apply, in combination with a convergence hypothesis, to get an active set identification result.  We then prove, in the nonconvex case, a novel $O(1/\sqrt{k})$ convergence rate result and active set identification for different stepsizes (under suitable assumptions on the set of stationary points). By exploiting those results, we also give explicit active set complexity bounds for both strongly convex and nonconvex objectives. While we initially consider the probability simplex as feasible set, in the appendix we show how to adapt some of our results to generic polytopes.  
 \end{abstract}



    \pagestyle{myheadings}
    \thispagestyle{plain}
    \markboth{I. M. BOMZE, F. RINALDI, D. ZEFFIRO}{ACTIVE SET COMPLEXITY OF THE AFW ALGORITHM}


\section{Introduction}
Identifying a surface containing a solution (and/or the support of sparse solutions) represents a 
relevant task in optimization, since it allows to reduce the dimension of the problem at handle and to apply 
a more sophisticated method in the end (see, e.g. \cite{bertsekas:1982,birgin:2002,2017arXiv170307761C,cristofari:2018new,desantis:2012,hager2011gradient,hager:2006,hager2016active}). 
This is the reason why, in the last decades, identification properties 
of optimization methods have been the subject of extensive studies.

The Frank-Wolfe (FW) algorithm, first introduced in \cite{frank1956algorithm}, is a classic  first-order optimization method that  has  
recently re-gained  popularity thanks to the way it can easily handle the structured constraints appearing in many real-world applications. 
This method and its variants have been indeed applied in the context of, e.g., submodular optimization problems \cite{bach2013learning}, 
variational inference problems \cite{krishnan2015barrier} and sparse neural network training \cite{grigas2019stochastic}. 
It is important to notice that the FW approach has a relevant drawback with respect to other algorithms:  even when dealing
with the  simplest polytopes, it cannot identify the active set in finite time (see, e.g., \cite{FOIMP}). 
Due to the renewed interest in the method, it has hence become a relevant issue to determine whether 
some FW variants admit active set identification properties similar to those of other first order methods. 
In this paper we focus on the  away-step  Frank-Wolfe (AFW) method  and analyze active set identification properties
for problems of the form
\begin{equation*} 
    \tx{min} \left\{f(x) \ | \ x \in \Delta_{n - 1} \right\},
    \end{equation*}
where the objective $f$ is a 
differentiable function with Lipschitz regular gradient and the feasible set   
$$\Delta_{n - 1}=\left\{x\in \R^n: \,\displaystyle \sum_{i=1}^n x_i = 1, \, x \ge 0\right\}$$ 
is the probability simplex. We further extend some of the active set 
complexity results to general polytopes.\\
\subsection{Contributions} It is a classic result that on polytopes and under strict complementarity conditions 
the AFW with exact linesearch identifies the face containing the minimum in finite time for strongly convex 
objectives \cite{guelat1986some}. More general active set identification properties for Frank-Wolfe variants 
have recently been  analyzed in \cite{FOIMP}, where the authors proved active set identification for sequences convergent 
to a stationary point, and  AFW convergence to a stationary point for $C^2$ objectives with a finite number 
of stationary points and satisfying a technical convexity-concavity assumption (this assumption is substantially a generalization 
of a property related to quadratic possibly indefinite functions). 
The main contributions of this article with respect to \cite{FOIMP} are twofold: \\
\begin{itemize}
 \item First, we give quantitative local and global active set identification complexity bounds
under suitable assumptions on the objective. The key element in the computation of those bounds 
is a quantity that we call "active set radius". This radius determines a neighborhood of a stationary point for which the AFW at each iteration identifies an active constraint (if there is any not yet identified one). 
In particular, to get the active set complexity bound it is sufficient to know how many iterations it takes 
for the AFW sequence to enter this neighborhood. \\
\item Second, we analyze the identification properties of AFW without the technical 
convexity-concavity $C^2$ assumption used in \cite{FOIMP} (we consider general nonconvex objectives with Lipschitz gradient instead). 
More specifically, we prove active set identification under different conditions on the stepsize and some additional hypotheses on the support of  stationary points. \\
\end{itemize}
In order to prove our results, we consider stepsizes dependent on 
the Lipschitz constant of the gradient (see, e.g., \cite{balashov2019gradient}, \cite{iusem2003convergence} and references therein). 
By exploiting the affine invariance property  of the AFW (see, e.g., \cite{jaggi2013revisiting}), we also extend some of the results to generic polytopes. In our  analysis we will see how the AFW identification properties are related to the value of Lagrangian multipliers on stationary points. 
This, to the best of our knowledge, is the first time that some active set complexity bounds are given for a variant of the FW algorithm.

The paper is organized as follows:
after presenting the AFW method and the setting in Section~\ref{prel}, we study the local behaviour of this algorithm regarding the active set in Section~\ref{asradius}. In Section~\ref{asbds} we provide active set identification results in a quite general context, and apply these to the strongly convex case for obtaining complexity bounds. Section~\ref{S:nonconv} treats the nonconvex case, giving both global and local active set complexity bounds. In the final Section~\ref{concl} we draw some conclusions. To improve readability, some technical details are deferred to an appendix.

\subsection{Related work} In \cite{burke1988identification}  the authors proved that the projected gradient method and other converging 
sequential quadratic programming methods identify quasi-polyhedral faces under some nondegeneracy conditions. 
  In \cite{burke1994exposing} those results were extended to the case of exposed faces in polyhedral sets without the nondegeneracy assumptions. 
  This extension is particularly relevant to our work since the identification of exposed faces in polyhedral sets is 
  the framework that we will use in studying the AFW on polytopes. In \cite{wright1993identifiable} the results of 
  \cite{burke1988identification} were generalized to certain nonpolyhedral surfaces called "$C^p$ identifiable" contained
  in the boundary of convex sets.  A key insight in these early works was the openness of a generalized normal cone defined 
  for the identifiable surface containing a nondegenerate stationary point. This openness guarantees that, in a neighborhood of 
  the stationary point, the projection of the gradient identifies the related surface. It turns out that for linearly constrained 
  sets the generalized normal cone is related to positive Lagrangian multipliers on the stationary point.\\ 
  A generalization of \cite{burke1988identification} to nonconvex sets was proved in \cite{burke1990identification}, while an extension 
  to nonsmooth objectives was first proved in \cite{hare2004identifying}. Active set identification results have also been proved for 
  a variety of projected gradient, proximal gradient and stochastic gradient related methods (see for instance \cite{sun2019we} and 
  references therein). \\ 
 Recently, explicit active set complexity bounds have been given for some of the methods listed above. 
 Bounds for proximal gradient and block coordinate descent method were analyzed in \cite{nutini2019active} and \cite{nutini2017let} 
 under strong convexity assumptions on the objective. 
 A more systematic analysis covering many gradient related proximal methods
 (like, e.g., accelerated gradient, quasi Newton and stochastic gradient proximal methods)
 was carried out in \cite{sun2019we}. \\
 As for FW-like methods, in addition to the results in \cite{guelat1986some} and \cite{FOIMP} discussed earlier, identification 
 results have been proved in \cite{clarkson2010coresets} for  fully corrective  variants on the probability simplex. However, since fully 
 corrective variants require to compute the minimum of the objective on a given face at each iteration, they are not suited 
 for nonconvex problems. 
\section{Preliminaries}\label{prel}
In the rest of this article $f: \Delta_{n - 1} \rightarrow \R$ will be a function with gradient having Lipschitz constant 
$L$ and $\mathcal{X}^*$ will be the set of stationary points of $f$. The constant $L$ will also be used as Lipschitz constant for 
$\nabla f$ with respect to the norm $\n{\cdot}_1$. This does not require any additional hypothesis on $f$ since  $\n{\cdot}_1 \geq \n{\cdot}$, so that
\begin{equation*}
\n{\nabla f(x) - \nabla f(y)} \leq L \n{x-y} \leq L \n{x- y }_1
\end{equation*}
for every $x, y \in \Delta_{n - 1}$. \\
For $x \in \R^n$, $X \subset \R^n$ the function $\textnormal{dist}(x, X)$ will be the standard point 
set distance and for $A \subset \R^n$ the function $\textnormal{dist}(A, X)$ will be the minimal distance between points in the sets: 
\begin{equation*}
\textnormal{dist}(A, X) = \inf_{a \in A, x \in X}  \|a-x\| \, .
\end{equation*}  
We define $\textnormal{dist}_1$ in the same way but with respect to $\| \cdot \|_1$. We use the notation 
\begin{equation*}
	\supp(x) = \{i \in [1:n] \ | \ x_i \neq 0 \}
\end{equation*}
for the support of a point $x \in \mathbb{R}^n$. \\
Given a (convex and bounded) polytope 
$P$ and a vector $c$ we define the face of $P$ \textit{exposed by} $c$ as 
\begin{equation*}
\mathcal{F}(c) =\textnormal{argmax}\{c^\top x \ | \ x \in P \} \ .
\end{equation*}
It follows from the definition that the face of $P$ exposed by a linear function is always unique and nonempty. \\
We now introduce the multiplier functions, which were recently used in \cite{2017arXiv170307761C} to define an active
set strategy for minimization over the probability simplex.\\
For every $x \in \Delta_{n-1}$, $i \in [1:n]$ the multiplier function $\lambda_i: \Delta_{n - 1} \rightarrow \R$
is defined as $$ \lambda_i(x) = \nabla f(x)^\top(e_i - x), $$
or in vector form
\begin{equation*}
\lambda(x) = \nabla f(x) - x^\top\nabla f(x)e\ .
\end{equation*}
For every $x \in \mathcal{X}^*$ these functions coincide with the Lagrangian multipliers of the constraints $x_i \geq 0$.  \\
For a sequence $\{a_k\}_{k \in \mathbb{N}_0}$ we will drop the subscript and write simply $\{a_k\}$ (unless of course the sequence is defined on some other index set). \\
FW variants require a linear minimization oracle for the feasible set (the probability simplex in our case):
\begin{equation*}
    \tx{LMO}_{\Delta_{n-1}}(r) \in \argmin \{ x^\top r \ | \ x \in \Delta_{n-1}  \}.
\end{equation*}
Keeping in mind that  
$$\Delta_{n-1}=\tx{conv}(\{e_i,\ i=1,\dots,n \}),$$
we can assume that $\tx{LMO}_{\Delta_{n-1}}(r)$ always returns a vertex of the probability simplex, that
is 
$$\tx{LMO}_{\Delta_{n-1}}(r) = e_{\hat \imath}$$
with $\hat \imath \in \displaystyle\argmin_{i} r_i.$ 
 
Algorithm 1 is the classical FW method on the probability simplex. At each iteration, this first order method generates a descent direction 
that points from the current iterate $x_k$ to a vertex $s_k$ minimizing the scalar product with the gradient, and then moves along 
this search direction of a suitable stepsize if stationarity conditions are not satisfied. 
\vspace{3mm}
\begin{center}
	    \begin{tabular}{|l|}
	       \hline 
	       \textbf{Algorithm 1} Frank--Wolfe method on the probability simplex \\
	       \hline
	     1. \textbf{Initialize} $x_0 \in \Delta_{n-1}$, $k := 0$  \\
		2. Set $s_k :=e_{\hat\imath},$ with $\hat\imath\in\displaystyle\argmin_{i} \nabla_i f(x_k)$ and $d_k^{\mathcal{FW}} := s_k - x_k$ \\
		3. If $x_k$ is stationary, then STOP  \\
		4. Choose the step size $\alpha_k \in (0, 1]$ with a suitable criterion  \\
		5. Update: $x_{k+1} := x_k + \alpha_k d^{\mathcal{FW}}_k$ \\
		6. Set $k := k+1$. Go to Step 2. \\
		\hline
	    \end{tabular}
\end{center}
\vspace{3mm}
It is well known \cite{canon1968tight,wolfe1970convergence} that 
the method exhibits a zig zagging behaviour as the sequence of iterates $\{x_k\}$ approaches a solution on the 
boundary of the feasible set. In particular, when this happens the sequence $\{x_k\}$ converges 
slowly and, as we already mentioned, it does not identify the smallest face containing the solution in finite time. 
Both of these issues are solved by the away-step variant of the FW method, reported in Algorithm 2.
The AFW at every iteration chooses between the classic FW direction and the away-step direction $d_k^{\mathcal{A}}$ 
calculated at Step 4. This away direction shifts weight away from the worst vertex to the other vertices 
used to represent the iterate $x_k$.
Here the worst vertex (among those having 
positive weight in the iterate representation) is the one with the greatest scalar product with the gradient, or, equivalently, 
the one that maximizes the linear approximation of $f$ given by $\nabla f(x_k)$. The stepsize upper bound $\alpha_{k}^{\max}$ in Step 8 is the maximal 
possible for the away direction given the boundary conditions. When the algorithm performs an away step, we have that either the support 
of the current iterate stays the same or decreases of one (we get rid of the component whose index is associated to the away direction in case $\alpha_k = \alpha_k^{\max}$). 
On the other hand, when the algorithm performs a Frank Wolfe step, only the vertex given by the $\tx{LMO}$ is eventually added to the support of 
the current iterate. These two properties are fundamental for the active set identification of the AFW.
\vspace{3mm}
\begin{center}
	    \begin{tabular}{|l|}
	       \hline 
	       \textbf{Algorithm 2} Away--step Frank--Wolfe on the probability simplex \\
	       \hline
	     1. \ \textbf{Initialize} $x_0 \in \Delta_{n-1}$, $k := 0$  \\
		2. \ Set $s_k :=e_{\hat\imath},$ with $\hat\imath\in\displaystyle\argmin_{i} \nabla_i f(x_k)$ and $d_k^{\mathcal{FW}} := s_k - x_k$ \\
		3. \ If $x_k$ is stationary then STOP  \\ 
		4. \ Let $v_k :=e_{\hat\jmath},$ with $\hat\jmath\in\displaystyle\argmax_{j\in S_k} \nabla_j f(x_k)$, $S_k:=\{j: (x_k)_j > 0 \}$ and $d_k^{\mathcal{A}} := x_k-v_k$ \\
		5. \ If  $-\nabla f(x_k)^\top d_k^{\mathcal{FW}} \geq -\nabla f(x_k)^\top d_k^{\mathcal{A}}$ then \\
		6. \ \quad $d_k := d_k^{\mathcal{FW}}$, and  $\alpha_{k}^{\max} :=1$ \\
		7. \ else \\
		8. \ \quad $d_k := d_k^{\mathcal{A}}$, \text{and} $\alpha_k^{\max} := (x_k)_i/(1-(x_k)_i) $ \\
		9. \ End if \\
		10. Choose the step size $\alpha_k \in (0, \alpha_k^{\max}]$ with a suitable criterion  \\
		11. Update: $x_{k+1} := x_k + \alpha_k d_k$ \\
		12. $k := k+1$. Go to step 2. \\
		\hline
	    \end{tabular}
\end{center}
\vspace{3mm}
In our analysis, we will sometimes require a lower bound on the step size which is always satisfied by the exact linesearch and the Armijo rule for a proper choice of the parameters. 
\section{Local active set variables identification property of the AFW}\label{asradius}
In this section we prove a rather technical proposition which is the key tool to give quantitative estimates for the active set complexity. It states that when the sequence is close enough to a fixed stationary point at every step the AFW identifies one variable violating the complementarity conditions with respect to the multiplier functions on this stationary point (if it exists), and it sets the variable to $0$ with an away step. The main difficulty is giving a tight estimate for how close the sequence must be to a stationary point for this identifying away step to take place. \\
A lower bound on the size of the nonmaximal away steps is needed in the following theorem, otherwise of course the steps could be arbitrarily small and there could be no convergence at all. \\
Let $\{x_k\}$ be the sequence of points generated by the AFW.
We further indicate with  $x^*$ a fixed point in $\mathcal{X}^*$, with {\em the extended support}  $$I = \{i \in [1:n] \ | \ \lambda_i (x^*) = 0 \}$$ and with $I^c = \{ 1, ...n \}\sm I$.
Note that by complementary slackness, we have $x_j^*=0$ for all $j\in I^c$.

Before proving the main theorem we need to prove the following lemma to bound the Lipschitz constant of the multipliers on stationary points.
\begin{Lemma} \label{lipest}
	Given $h>0$, $x_k \in \Delta_{n - 1}$ such that $\|x_k - x^*\|_1 \leq h$ let $$O_k = \{i \in I^c \ | \ (x_k)_i = 0\}$$ and assume that $O_k \neq I^c$. Let $\delta_{k} = \max_{i \in [1:n]\sm O_k} \lambda_i(x^*)$. For every $i \in \{1, ..., n\}$: 
	\begin{equation} \label{lip}
	| \lambda_i (x^*) - \lambda_i(x_k) | \leq h(L + \frac{\delta_k}{2})\ .
	\end{equation}	
\end{Lemma}
\begin{proof}
By considering the definition of $\lambda(x)$, we can write
	\begin{eqnarray} \label{split}
	|{\lambda}_i(x_k) - {\lambda}_i(x^*)| &=& |\nabla f(x_k)_i - \nabla f (x^*)_i + \nabla f(x^*)^\top (x^*-x_k) + (\nabla f(x^*)-\nabla f(x_k))^\top x_k| \nonumber\\
	&\leq& |\nabla f(x^*)_i - \nabla f (x_k)_i  + (\nabla f(x_k)-\nabla f(x^*))^\top x_k| + | \nabla f(x^*)^\top(x^*-x_k)|\ .
	\end{eqnarray}
	By taking into account the fact that $x_k \in \Delta_{n-1}$ and gradient of $f$ is Lipschitz continuous, we have 
	\begin{eqnarray} \label{1piece}
  |\nabla f(x_k)_i - \nabla f (x^*)_i  + (\nabla f(x^*)-\nabla f(x_k))^\top x_k| &=& |(\nabla f(x^*)-\nabla f(x_k))^\top (x_k-e_i) | \nonumber\\ 
	                                                                            &\leq& \|\nabla f(x^*)-\nabla f(x_k)\|_1 \| x_k-e_i\|_{\infty}\\
	                                                                            &\leq& Lh,\nonumber
	\end{eqnarray}   
	where the last inequality is justified by  the H\"older inequality with exponents $1, \infty$. \\
	We now bound the second term in the right-hand side of \eqref{split}. Let 
	\begin{equation*}
	u_j = \max\{0, (x^*-x_k)_j\}, \ l_j = \max\{0, -(x^*-x_k)_j\}\, . 
	\end{equation*}
We have	$\sum_{j \in [1:n]} x^*_j = \sum_{j \in [1:n]} (x_k)_j = 1 $ since $\{x^*, x_k\}\subset \in \Delta_{n-1}$, so that
	\begin{equation*}
	\sum_{j \in [1:n]} (x^*-x_k)_j =  \sum_{j \in [1:n]} (u_j - l_j) = 0 \quad\mbox{and hence}\quad \sum_{j \in [1:n]} u_j = \sum_{i \in [1:n]} l_j.
	\end{equation*}
	Moreover, $ h'\myeq 2\sum_{j \in [1:n]} u_j = 2\sum_{j \in [1:n]} l_j = \sum_{j \in[1:n]} u_j + l_j  = \sum_{j \in [1:n]} |x^*_j - (x_k)_j| \leq h$, hence 
	\begin{equation*}
	h'/2 = \sum_{j \in [1:n]} u_j = \sum_{j \in [1:n]} l_j   \leq h/2\ .
	\end{equation*}
	We can finally bound the second piece of \eqref{split}, using $u_j=l_j=0$ for all $j\in O_k$ (because $(x_k)_j =x_j^*=0$):
	\begin{eqnarray} \label{2piece}
	| \nabla f(x^*)^\top (x^*-x_k)|&=& | \nabla f(x^*)^\top k  - \nabla f(x^*)^\top l |\leq \frac{h'}{2}(\nabla f(x^*)_M - \nabla f(x^*)_m) \nonumber\\
	                          &\leq& \frac{h}{2}(\nabla f(x^*)_M - \nabla f(x^*)_m),
	\end{eqnarray}
	where $\nabla f(x_k)_M$ and $\nabla f(x_k)_m$ are respectively the maximum and minimum component of the gradient in $[1:n]\setminus O_k$.\\
	Now, considering inequalities  \eqref{split}, \eqref{1piece} and \eqref{2piece}, we can write
	\begin{equation*}
	|{\lambda}_i(x_k) - {\lambda}_i(x^*)|\leq Lh+ \frac{h}{2}(\nabla f(x^*)_M - \nabla f(x^*)_m).
	\end{equation*}
	By taking into account the definition of $\delta_k$ and the fact that $\lambda(x^*)_j\geq 0$ for all $j$, we can write  
	$$\delta_{k} = \max_{i,j \in [1:n]\sm O_k} (\nabla f(x^*)_i-\nabla f(x^*)_j)\geq \nabla f(x^*)_M-\nabla f(x^*)_m.$$
	We can finally write
	$$
	|{\lambda}_i(x_k) - {\lambda}_i(x^*)|\leq h(L+\frac{\delta_k}{2}),
	$$
thus concluding the proof.
\end{proof}
We now show a few simple but important results that connect the multipliers and the directions selected by the AFW algorithm. Notice that for a fixed $x_k$ the multipliers $\lambda_i(x_k)$ are the values of the linear function $x \mapsto \nabla f(x_k)^\top x$ on the vertices of $\Delta_{n - 1}$ (up to a constant), which in turn are the values considered in the AFW to select the direction. This basic observation is essentially everything we need for the next results. 
\begin{Lemma} \label{awstep}
Let $S_k = \{ i \in \{1, ..., n\} \ | \ (x_k)_i > 0 \}$. Then
	\begin{itemize}
		\item[(a)] If $\max\{\lambda_i(x_k) \ | \ i \in S_k \} > \max\{- \lambda_i(x_k) \ | \ i \in [1:n] \}$, then the AFW performs an away step with $d_k = d_k^{\mathcal{A}} = x_k - e_{\hat\imath}$ for some 
		$i \in \textnormal{argmax} \{\lambda_i(x_k) \ | \ i \in S_k \}$. 
		\item[(b)] For every $i \in [1:n] \sm  S_k$ if $\lambda_i(x_k) > 0$ then $(x_{k+1})_i =(x_k)_i = 0$.  
	\end{itemize}
\end{Lemma}
\begin{proof}
	(a) Notice that since the vertices of the probability simplex are linearly independent for every $k$ the set of active atoms is necessarily $S_k$. In particular \\
	$d_k^{\mathcal{A}} \in \textnormal{argmax} \{ -\nabla f(x_k)^\top d) \ | \ d = x_k - e_i, i \in S_k \}$ and this implies
	\begin{equation} \label{awdir}
	d_k^{\mathcal{A}} = x_k - e_{\hat\imath} \quad \textnormal{for some } \hat\imath \in \textnormal{argmax}\{-\nabla f(x_k)^\top (x_k - e_i) \ | \ i \in S_k \} = \textnormal{argmax} \{\lambda_i(x_k) \ | \ i \in S_k \}\ .
	\end{equation}	
	As a consequence of \eqref{awdir} 
	\begin{equation} \label{lambda1}
	- \nabla f(x_k)^\top d_k^{\mathcal{A}}  = \max \{- \nabla f(x_k)^\top d \ | \  d= x_k- e_i, i \in S_k \} =  \max \{\lambda_i(x_k) \ | \  i \in S_k \}\ ,	
	\end{equation}	
	where the second equality follows from $\lambda_i(x_k) = -\nabla f(x_k)^\top  d$ with $d =  x_k-e_i$. \\ 
	Analogously
	\begin{equation} \label{lambda2} 
	\begin{aligned}
	- \nabla f(x_k)^\top d_k^{\mathcal{FW}} & = \max \{ - \nabla f(x_k)^\top d \ | \  d = e_i - x_k, i \in \{1, ...n \} \} = \\ & = \max \{ - \lambda_i(x_k) \ | \ i \in \{1, ...n \} \}\ .  	
	\end{aligned}
	\end{equation}	
	We can now prove that $ -\nabla f(x_k)^\top d_k^{\mathcal{FW}} < - \nabla f(x_k)^\top d_k^{\mathcal{A}}$, so that the away direction is selected under assumption (a):
	\begin{align*}
	& -\nabla f(x_k)^\top d_k^{\mathcal{FW}} = \max \{ - \lambda_i(x_k) \ | \ i \in \{1, ...n \} \} < \\ & < \max \{\lambda_i(x_k) \ | \  i \in S_k \} = - \nabla f(x_k)^\top d_k^{\mathcal{A}}, 	 
	\end{align*}
	where we used \eqref{lambda1} and \eqref{lambda2} for the first and the second equality respectively, and the inequality is true by hypothesis. \\
	(b) By considering the fact that $(x_k)_i=0$, we surely cannot choose the vertex $e_i$ to define the away-step direction. Furthermore, since $\lambda(x_k)_i=\nabla f(x_k)^\top (e_i-x_k)>0$,
	direction $d=e_i-x_k$ cannot be chosen as the Frank-Wolfe direction at step $k$ as well. This guarantees that $(x_{k+1})_i=0$.
\end{proof}
We can now prove the main theorem. The strategy will be to split $[1:n]$ in three subsets $I$, $J_k \subset I^c$ and $O_k = I^c \sm  J_k$ and use Lemma $\ref{lipest}$ to control the variation of the multiplier functions on each of these three subsets. 
In the proof we examine two possible cases under the assumption of being close enough to a stationary point. If $J_k = \emptyset$, which means that the current iteration of the AFW has identified the support of the stationary point, then we will show that the AFW chooses a direction contained in the support, so that also $J_{k+1} = \emptyset$.\\
If $J_k \neq \emptyset$, we will show that in the neighborhood claimed by the theorem the largest multiplier in absolute value is always positive, with index in $J_k$, and big enough, so that the corresponding away step is maximal. This means that the AFW at the iteration $k+1$ identifies a new active variable.  

\begin{Th} \label{ascmain}
If $I^c$ is not the empty set,
	let us define
	$$\delta_{\min} = \min\{\lambda_i(x^*) \ | \ i \in I^c \} > 0, \ J_k =  \{i \in I^c \ | \ (x_k)_i >0 \}\ .$$ Assume that for every $k$ such that $d_k = d_k^{\mathcal{A}}$ the step size $\alpha_k$ is either maximal with respect to the boundary condition (that is $\alpha_k = \alpha_k^{\max}$) or $\alpha_k \geq \frac{- \nabla f(x_k)^\top d_k}{L\|d_k\|^2} $. If $\| x_k - x^*\|_1 <\frac{\delta_{\min}}{\delta_{\min}+ 2L} = r_*$ then 
	\begin{equation} \label{actset}
	|J_{k+1}| \leq \max\{0, |J_k| - 1\}\ .
	\end{equation}
	The latter relation also holds in case $I^c=\emptyset$ whence we put $r_* = +\infty$.
\end{Th}

\begin{proof}
	If $ I^c = \emptyset$, or equivalently, if $\lambda(x^*) = 0$, then there is nothing to prove since $J_k \subset I^c = \emptyset \Rightarrow	|J_k|= |J_{k+1}|= 0$. \\
	So assume $I^c \neq \emptyset$. By optimality conditions $\lambda_i(x^*) \geq 0$ for every $i$, so necessarily $\delta_{\min} > 0$. \\
	For every $i \in [1:n]$, by Lemma \ref{lipest} 
	\begin{equation} \label{lambdaineq}
	\begin{aligned}
	\lambda_i(x_k) 
	&\geq  \lambda_i(x^*) - \|x_k-x^*\|_1(L + \frac{\delta_k}{2}) >\\ &>\lambda_i(x^*) - r_*(L + \frac{\delta_k}{2}) =\lambda_i(x^*) - \frac{\delta_{\min}(L + \frac{\delta_k}{2})}{2L  + \delta_{\min}}\ .
	\end{aligned}
	\end{equation}
	We now distinguish two cases. \\
	\textbf{Case 1:} $|J_k| = 0$. Then $\delta_k = 0$ because $J_k \cup I = I$ and $\lambda_i(x^*) = 0$ for every $i \in I$. Relation \eqref{lambdaineq} becomes 
	\begin{equation*}
	\lambda_i(x_k) \geq\lambda_i(x^*) - \frac{\delta_{\min}L}{2L  + \delta_{\min}}, 
	\end{equation*}
	so that for every $i \in I^c$, since $\lambda_i(x^*) \geq \delta_{\min}$, we have
	\begin{equation} \label{lambdaik}
	\lambda_i(x_k) \geq \delta_{\min} - \frac{\delta_{\min}L}{2L  + \delta_{\min}} > 0\ .
	\end{equation}
	This means that for every $i \in I^c$ we have $(x_k)_i = 0$ by the Case 1 condition $J_k = \emptyset$ and $\lambda_i(x_k) > 0$ by \eqref{lambdaik}. We can then apply part (b) of Lemma \ref{awstep}  and conclude $(x_{k+1})_i = 0$ for every $i \in I^c$. Hence $J_{k+1}= \emptyset = J_k $ and Theorem \ref{ascmain} is proved in this case. \\	
	\textbf{Case 2.} $|J_k| > 0$.
	For every $i \in \textnormal{argmax}\{\lambda_j(x^*) \ | \ j \in J_k\}$, we have
	$$\lambda_i(x^*) = \max_{j \in J_k} \lambda_j(x^*) =  \max_{j \in J_k\cup I} \lambda_j(x^*), $$
	where we used the fact that $\lambda_j(x^*) = 0 < \lambda_i(x^*)$ for every $j \in I$. Then by the definition of $\delta_k$, it follows 
	$$ \lambda_i(x^*) = \delta_k. $$
	Thus \eqref{lambdaineq} implies
	\begin{equation} \label{first}
	\begin{aligned}
		\lambda_i(x_k) > \lambda_i(x^*) - \frac{\delta_{\min}(L + \frac{\delta_k}{2})}{2L  + \delta_{\min}} = \delta_k - \frac{\delta_{\min}(L + \frac{\delta_k}{2})}{2L  + \delta_{\min}}, 
	\end{aligned}
	\end{equation} 
	where we used \eqref{lambdaineq} in the inequality. 
	But since $\delta_k \geq \delta_{\min}$ and the function $y \mapsto - \frac{y}{2L + y}$ is decreasing in $\R_{> 0}$ we have
	\begin{equation} \label{first2}
		\delta_k - \frac{\delta_{\min}(L + \frac{\delta_k}{2})}{2L  + \delta_{\min}}  \geq \delta_k - \frac{\delta_{k}(L + \frac{\delta_k}{2})}{2L  + \delta_{k}} = \frac{\delta_k}{2}\ .
	\end{equation}
	Concatenating $\eqref{first}$ with $\eqref{first2}$, we finally obtain 
	\begin{equation}\label{newlab}
		\lambda_i(x_k) > \frac{\delta_k}{2}\ . 
	\end{equation}
	We will now show that $d_k = x_k - e_{\hat\jmath}$ with $\hat\jmath \in J_k$. \\
	For every $j \in I$, since $\lambda_j(x^*) = 0$, again by Lemma \ref{lipest}, we have
	\begin{equation} \label{Ibound}
	\begin{aligned}
	|\lambda_j(x_k)| &= |\lambda_j(x_k) - \lambda_j(x^*)| \leq  \|x_k-x^*\|_1(L + \delta_k/2) < \\
	 & < r_*(L + \delta_k / 2) = \frac{\delta_{\min}(L + \frac{\delta_k}{2})}{2L  + \delta_{\min}} \leq \delta_k/2,
	\end{aligned}
	\end{equation}
	where we used $\n{x_k - x^*}_1 < r_* $, which is true by definition, in the first inequality, and rearranged \eqref{first2} to get the last inequality. 
	For every $j \in I^c$, by \eqref{lambdaineq}, we can write
	\begin{equation*}
	\lambda_j(x_k) > \delta_{\min} - \frac{\delta_{\min}(L + \frac{\delta_k}{2})}{2L  + \delta_{\min}} > - \frac{\delta_k}{2}\ .
	\end{equation*}
	Then using this together with \eqref{Ibound} and \eqref{first}, we get $- \lambda_j (x_k) < \delta_k/2 < \lambda_h(x_k)$ for every $j \in [1:n], h \in \textnormal{argmax}\{\lambda_q(x^*) \ | \ q \in J_k \}$. So the hypothesis of Lemma \ref{awstep} is satisfied and $d_k = d_k^{\mathcal{A}} = x_k - e_{\hat\jmath}$ with $\hat\jmath \in \textnormal{argmax}\{\lambda_j(x_k) \ | \ j \in S_k \}$. 
	We need to show $\hat\jmath \in J_k$. But $S_k\subseteq I \cup J_k$ and by  \eqref{Ibound} if $\hat\jmath \in I$ then $\lambda_l(x_k) < \delta_k/2 < \lambda_j(x_k)$ for every $j \in \textnormal{argmax}\{\lambda_j(x^*) \ | \ j \in J_k\}$. If $\hat\jmath \in O_k$ then $(x_k)_{\hat\jmath} = 0$ and $\hat\jmath \notin S_k$.
	Hence we can conclude $\textnormal{argmax}\{\lambda_j(x_k) \ | \ j \in S_k \} \subseteq J_k$ and $d_k = x_k - e_{\hat\jmath}$ with $\hat\jmath \in J_k$. In particular, by \eqref{newlab} we get
	\begin{equation} \label{lamb>del}
	 \max\{\lambda_j(x_k) \ | \ j \in J_k \}=\lambda_{\hat\jmath}(x_k)  > \frac{\delta_k}{2}\, .
	\end{equation}
	We now want to show that $\alpha_k = \alpha_k^{\max}$. Assume by contradiction $\alpha_k < \alpha_{\max}$. Then by the lower bound on the stepsize and~\eqref{newlab}
	\begin{equation}\label{alphai}
	\begin{aligned}
	\alpha_k \geq \frac{-\nabla f(x_k)^\top d_k}{L\|d_k\|^2} = \frac{\lambda_i(x_k)}{L\|d_k\|^2} \geq \frac{\delta_{\min}}{2L\|d_k\|^2}\, ,
	\end{aligned}
	\end{equation}
	where in the last inequality we used \eqref{lamb>del} together with $\delta_k \geq \delta_{\min}$. Also, by Lemma \ref{simpelem}
	\begin{equation} \label{dk<}
	\begin{aligned}
	\|d_k \| & = \| e_{\hat\jmath} - x_k\| \leq \sqrt{2} (e_{\hat\jmath} - x_k)_{\hat\jmath} = -\sqrt{2}(d_k)_{\hat\jmath} \Rightarrow \frac{(d_k)_{\hat\jmath}}{\|d_k\|^2} \leq\frac{(d_k)_{\hat\jmath}}{\|d_k\|\sqrt{2}}\leq - 1/2\\
	(x_k)_{\hat\jmath} &= (x_k - x^*)_{\hat\jmath} \leq \frac{\|x_k - x^*\|_1}{2} < \frac{r_*}{2} = \frac{\delta_{\min}}{4L + 2\delta_{\min}}.
	\end{aligned}
	\end{equation}
	Finally, combining \eqref{dk<} with \eqref{alphai}
	\begin{align*}
	(x_{k+1})_{\hat\jmath} &= (x_k)_{\hat\jmath} + (d_k)_{\hat\jmath} \alpha_k < \frac{r_*}{2} -  \frac{\|d_k\|^2}{2} \alpha_k\leq \frac{r_*}{2} -  \frac{\|d_k\|^2}{2} \frac{\delta_{\min}}{2L\|d_k\|^2}\\ 
	&= \frac{\delta_{\min}}{4L + 2\delta_{\min}} - \frac{\delta_{\min}}{4L} < 0,
	\end{align*}
	where we used \eqref{alphai} to bound $\alpha_k$ in the first inequality, \eqref{dk<} to bound $(x_k)_{\hat\jmath}$ and $\frac{(d_k)_{\hat\jmath}}{\|d_k\|^2}$. Hence $(x_{k+1})_{\hat\jmath}< 0$, contradiction.  
\end{proof}

\section{Active set complexity bounds}\label{asbds}
Before giving the active set complexity bounds in several settings it is important to clarify that by active set associated to a stationary point $x^*$ we do not mean the set $\mathrm{supp}(x^*)^c = \{i \in [1:n] \ | \ (x^*)_i = 0\}\}$ but the set $I^c(x^*) = \{i \in [1:n] \ | \ \lambda_i(x^*) > 0 \}$. In general $I^c(x^*) \subset \supp(x^*)^c$ by complementarity conditions, with 
\begin{equation} \label{eq:supp}
\supp(x^*)^c = I^c(x^*)	\Leftrightarrow \tx{complementarity conditions are strict in $x^*$}.
\end{equation}
The face $\mathcal{F}$ of $\Delta_{n - 1}$ defined by the constraints with indices in $I^c(x^*)$ still has a nice geometrical interpretation: it is the face of $\Delta_{n-1}$ exposed by $-\nabla f(x^*)$. \\
It is at this point natural to require that the sequence $\{x_k\}$ converges to a subset $A$ of $\mathcal{X}^*$ for which $I^c$ is constant. This motivates the following definition:
\begin{Def} \label{support}
	A compact subset $A$ of $\mathcal{X}^*$ is said to have the {\em support identification property (SIP)} if there exists an index set $I^c_A \subset [1:n]$ such that 
	$$I^c (x)=I^c_A \quad \mbox{for all }x\in A\, .$$ 
\end{Def}
The geometrical interpretation of the above definition is the following: for every point in the subset $A$ the negative gradient $-\nabla f(x^*)$ exposes the same face. This is trivially true if $A$ is a singleton, and it is also true if for instance $A$ is contained in the relative interior of a face of $\Delta_{n - 1}$ and strict complementarity conditions hold for every point in this face. 
We further define $$ \delta_{\min}(A) =\min\{\lambda_i(x) \ | \ x \in A, \  i \in I^c_A\}\ . $$
Notice that by the compactness of $A$ we always have $\delta_{\min}(A) > 0$ if $A$ enjoys the SIP. We can finally give a rigorous definition of what it means to solve the active set problem:
\begin{Def}
	Consider an algorithm generating a sequence $\{x_k\}$ converging to a subset $A$ of $\mathcal{X}^*$ enjoying the SIP. We will say that this algorithm solves the active set problem in $M$ steps if  $(x_k)_i = 0$ for every $i \in I^c_A$, $k\geq M$. 
\end{Def} 
We can now apply Theorem \ref{ascmain} to show that once a sequence is definitely close enough to a set enjoying the SIP, the AFW identifies the active set in at most $|I^c|$ steps. We first need to define a quantity that we will use as a lower bound on the stepsizes:
\begin{equation}    \label{alphabound}
	\bar{\alpha}_k = \min\left(\alpha_k^{\max}, \frac{-\nabla f(x_k)^\top d_k }{{L\n{d_k}^2}} \right)\ ,
\end{equation}
\begin{Th} \label{activecompl}
	Let $\{x_k\}$ be a sequence generated by the AFW, with stepsize $\alpha_k \geq \bar{\alpha}_k$. Let $\mathcal{X}^*$ be the set of stationary points of a function $f: \Delta_{n - 1} \rightarrow \mathbb{R}$ with $\nabla f$ 
	having Lipschitz constant $L$. Assume that 
	 there exists a compact subset $A$ of $\mathcal{X}^*$ with the SIP such that $x_k \rightarrow A$.
Then there exists $M$ such that 
$$(x_k)_i = 0\quad\mbox{ for every }k \geq M\mbox{ and all }i \in I^c_A\,.$$ 
\end{Th}
\begin{proof}
	Let $J_k = \{i \in I^c_A \ | \ (x_k)_i> 0\}$ and choose $\bar{k}$ such that $\textnormal{dist}_1(x_k, A) < \frac{\delta_{\min}(A)}{2L + \delta_{\min}(A)} = r_*$ for every $k \geq \bar{k}$.
	Then for every $k \geq \bar{k}$ there exists $y^* \in A$ with $\|x_k - y^*\|_1 < r_*$. 
	But since by hypothesis for every $y^* \in A$ the support of the multiplier function is $I^c_A$ with $\delta_{\min}(A)\leq \lambda_i(y^*)$ for every $i \in I^c_A$, we can apply Theorem \ref{ascmain} with $y^*$ as fixed point and obtain that 
	$|J_{k+1}| \leq \max (0, |J_k| - 1)$. This means that it takes at most $|J_{\bar{k}}| \leq |I^c_A|$ steps for all the variables with indices in $I^c_A$ to be 0. Again by \eqref{actset},  we conclude by induction $|J_k| = 0$ for every $k\geq M=\bar{k}+|I^c_A|$, since $|J_{\bar{k} + |I^c_A|}|= 0$. 
\end{proof} 
The proof above also gives a relatively simple upper bound for the complexity of the active set problem:
\begin{Prop} \label{activecomp}
	Under the assumptions of Theorem \ref{activecompl}, the active set complexity is at most 
	$$\min\{\bar{k} \in \mathbb{N}_0 \ | \ \textnormal{dist}_1(x_k, A) < r_* \forall k \geq \bar{k} \} + |I^c_A|, $$
	where $r_* = \frac{\delta_{\min}(A)}{2L + \delta_{\min}(A)}$. 
\end{Prop}
We now report an explicit bound for the strongly convex case, and analyze in depth the nonconvex case in Section \ref{S:nonconv}. 
From strong convexity of $f$, it is easy to see that the following inequality holds for every $x$ on $\Delta_{n-1}$:
	\begin{equation}\label{he}
	 f(x) \geq f(x^*)+ \frac{u_1}{2} \|x - x^* \|_1^2,
	\end{equation}
	 with $u_1>0$. 
\begin{Cor} \label{ssimplex}
	Let $\{x_k\}$ be the sequence of points generated by AFW with $\alpha_k \geq \bar{\alpha}_k$. 
	Assume that $f$ is strongly convex and let
	\begin{equation}\label{ratescc}
	 h_{k} \leq q^k h_0,
	\end{equation}
		with $q < 1$ and $h_k = f(x_k) - f_*$, be the convergence rate  
	related to AFW. 
	Then the active set complexity is 
	$$\max\left (0, \left\lceil\frac{\textnormal{ln}(h_0) - \textnormal{ln}(u_1 r_*^2/2)}{\textnormal{ln(1/q)}}\right\rceil\right) +|I^c|\ . $$
\end{Cor}
\begin{proof}
	Notice that by the linear convergence rate \eqref{ratescc}, and the fact that $q<1$, the number of steps needed 
	to reach the condition 
	\begin{equation} \label{hkleq}
	h_k \leq \frac{u_1}{2}r_*^2
	\end{equation}
	is at most 
	$$ \bar{k} = \max \left(0, \left\lceil\frac{\textnormal{ln}(h_0) - \textnormal{ln}(u_1r_*^2/2)}{\textnormal{ln}(1/q)}\right\rceil\right)\ . $$	
	We claim that if condition \eqref{hkleq} holds then it takes at most $|I^c|$ steps for the sequence to be definitely in the active set. \\ 
	Indeed if $q^k h_0 \leq \frac{u_1}{2}r_*^2$ then necessarily $x_k \in B_1(x^*, r_*)$ by \eqref{he}, and by monotonicity of the bound~\eqref{ratescc} we then have $x_{k+h} \in B_1(x^*, r_*)$ for every $h \geq 0$. Once the sequence is definitely in $B_1(x^*, r_*)$ by $\eqref{actset}$ it takes at most $|J_{\bar{k}}| \leq |I^c|$ steps for all the variables with indices in $I^c$ to be 0. To conclude, again by \eqref{actset} since $|J_{\bar{k} + |I^c|}|= 0$ by induction $|J_m| = 0$ for every $m\geq \bar{k}+|I^c|$.
\end{proof}
\begin{Rem}
We would like to notice that strong convexity of $f$  in Corollary \ref{ssimplex} might actually be replaced by 
condition given in \eqref{he} if we assume the linear rate \eqref{ratescc} (which may not hold in the nonconvex case). 
\end{Rem}
The proof of AFW active set complexity for generic polytopes in the strongly convex case requires additional theoretical results and is presented in the appendix.
\par\medskip\noindent
\section{Active set complexity for nonconvex objectives} \label{S:nonconv}
In this section, we focus on problems with nonconvex objectives. We first give a more explicit convergence rate for AFW in the  nonconvex case, then we prove a general active set identification result for the method. Finally, we analyze both local and global active set complexity bounds related to AFW. A fundamental element in our analysis will be the FW gap function $g: \Delta_{n - 1} \rightarrow \R $ defined as 
 \begin{equation*}
 g(x) = \max_{i \in [1:n]} \{-\lambda_i(x)\}\ .
\end{equation*}
We  clearly have $g(x) \geq 0$ for every $x \in \Delta_{n - 1}$, with equality iff $x$ is a stationary point. The reason why this function is called FW gap is evident from the relation
\begin{equation*}
g(x_k) = -\nabla f(x_k)^\top d^{\mathcal{FW}}_k.
\end{equation*} 
This is a standard quantity appearing in the analysis of FW variants (see, e.g., \cite{jaggi2013revisiting} ) and is computed for free at each iteration of a FW-like algorithm. 
In \cite{lacoste2016convergence}, the author uses the gap to analyze the convergence rate of the classic FW algorithm in the 
nonconvex case. More specifically,  a convergence rate of $O(\frac{1}{\sqrt{k}})$ is proved for the minimal FW gap up to iteration $k$:
\begin{equation*}
g^*_k = \min_{0 \leq i \leq k-1} g(x_i).
\end{equation*}
The results extend in a nice and straightforward way the ones reported in \cite{nesterov2018lectures} for proving the convergence of gradient methods in the nonconvex case. 
Inspired by the analysis of the AFW method for strongly convex objectives reported in \cite{pena2018polytope}, 
we now study the AFW convergence rate in the nonconvex case  with respect to the sequence $\{g^*_k \}$. 

In the rest of this section we assume that the AFW starts from a vertex of the probability simplex. This is not a restrictive assumption. By exploiting affine invariance one can indeed apply the same theorems to the AFW starting from $e_{n+1}$ for $\tilde{f}: \Delta_n \rightarrow \R$ satisfying
\begin{equation*}
     \tilde{f}(y) = f(y_1e_1+...y_ne_n+y_{n+1}p),
\end{equation*}
where $p \in \Delta_{n-1}$ is the desired starting point. 
We will discuss more in detail the invariance of the AFW under affine transformations in Section \ref{generalafw}.  

\subsection{Global convergence}
We start investigating the minimal FW gap, giving estimates of rates of convergence:
\begin{Th} \label{nonconvb}
	Let $f^* = \min_{x \in \Delta_{n - 1}} f(x)$, and let $\{x_k\}$ be a sequence generated by the AFW algorithm applied to $f$ on $\Delta_{n-1}$, with $x_0$ a vertex of $\Delta_{n-1}$. Assume that the stepsize $\alpha_k$ is larger or equal  than $\bar{\alpha}_k$ (as defined in \eqref{alphabound}), and that
	\begin{equation} \label{eq:rho}
	f(x_k) - f(x_k + \alpha_kd_k) \geq \rho\bar{\alpha}_k \left(-\nabla f(x_k)^\top d_k\right)
	\end{equation}
	for some fixed $\rho > 0$. 
	Then for every $T \in \mathbb{N}$
	$$	g_T^* \leq \max\left(\sqrt{\frac{4L (f(x_0) - f^*)}{\rho T}}, \frac{4(f(x_0) - f^*)}{T} \right)\  .$$
\end{Th}
\begin{proof}
	Let $r_k = -\nabla f(x_k)$ and $g_k = g(x_k)$. We distinguish three cases. \\
	
	\textbf{Case 1.} $\bar{\alpha}_k < \alpha^{\max}_k$. 
	Then $\bar{\alpha}_k =  \frac{-\nabla f(x_k)^\top d_k}{{L\n{d_k}^2}} $ and relation \eqref{eq:rho} becomes 
	\begin{equation*}
	f(x_k) - f(x_k + \alpha_kd_k) \geq \rho\bar{\alpha}_k r_k^\top d_k = \frac{\rho}{L{\n{d_k}}^2} (r_k^\top d_k)^2
	\end{equation*}
	and consequently
	\begin{equation} \label{c1}
	f(x_k) - f(x_{k+1}) \geq \frac{\rho}{  L \n{d_k}^2} (r_k^\top d_k)^2 \geq \frac{\rho}{ L \n{d_k}^2}g_k^2 \geq \frac{\rho g_k^2}{2L},
	\end{equation}	
	where we used $r_k^\top d_k \geq g_k$ in the second inequality and $\n{d_k} \leq \sqrt{2}$ in the third one. \\ As for $S_k$, by hypothesis we have either $d_k = d_k^{\mathcal{FW}}$ so that $d_k = e_i - x_k$ or $d_k = d_k^{\mathcal{A}} = x_k - e_i$ for some $i \in \s{n}$. In particular $S_{k+1}\subseteq S_k \cup \{i\}$ so that $|S_{k+1}| \leq |S_k| + 1$. \\  
	\textbf{Case 2:} $\alpha_k = \bar{\alpha}_k =  \alpha^{\max}_k = 1, d_k = d_k^{\mathcal{FW}}$. 
	By the standard descent lemma~\cite[Proposition 6.1.2]{bertsekas2015convex} applied to $f$ with center $x_k$ and $\alpha = 1$
	$$ f(x_{k +1}) = f(x_k + d_k) \leq f(x_k) + \nabla f(x_k)^\top d_k + \frac{L}{2}\|d_k\|^2\ .  $$
	Since by the Case 2 condition $\min \left(\frac{-\nabla f(x_k)^\top d_k }{\|d_k\|^2L}, 1\right) =  \alpha_k = 1 $ we have
	\begin{equation*}
	\frac{-\nabla f(x_k)^\top d_k}{\|d_k\|^2L} \geq 1 \, \mbox{, so }\quad -L\n{d_k}^2 \geq \nabla f(x_k)^\top d_k\, ,
	\end{equation*} 
	hence we can write
	\begin{equation} \label{c2}
	f(x_k) - f(x_{k+1}) \geq -\nabla f(x_k)^\top d_k - \frac{L}{2}\|d_k\|^2 \geq - \frac{\nabla f(x_k)^\top d_k}{2} \geq \frac{1}{2}g_k\ . 
	\end{equation}
	Reasoning as in Case 1 we also have $|S_{k +1}| \leq |S_k| + 1$. \\
	\textbf{Case 3:} $\alpha_k = \bar{\alpha}_k = \alpha^{\max}_k, \ d_k = d_k^{\mathcal{A}}$. Then $d_k = x_k - e_i $ for $i \in S_k$ and  $$(x_{k + 1})_j=  (1+\alpha_k)(x_k)_j - \alpha_k (e_i)_j,$$ 
	with $\alpha_k = \alpha^{\max}_k = \frac{(x_k)_i}{1 - (x_k)_i}$. Therefore $(x_{k+1})_j = 0$ for $j \in \s{n} \sm S_k \cup \{i\} $ and $(x_{k+1})_j \neq 0$ for $j \in S_k \sm \{i\}$. In particular $|S_{k+1}| = |S_k| - 1$. 
	\vspace{2mm}
	
	For $i = 1,2,3$ let now $n_i(T)$ be the number of Case $i$ steps done in the first $T$ iterations of the AFW. We have by induction on the recurrence relation we proved for $|S_k|$
	\begin{equation}\label{n1n2}
	|S_{T}| - |S_0| \leq n_1(T) + n_2(T) - n_3(T)\ ,
	\end{equation}
	for every $T \in \mathbb{N}$. \\
	Since $n_3(T) = T-n_1(T) - n_2(T)$ from \eqref{n1n2} we get
	\begin{equation*}
	n_1(T) + n_2(T) \geq  \frac{T + |S_{T}| -|S_0|}{2} \geq \frac{T}{2}\ ,
	\end{equation*}
	where we used $|S_0| = 1 \leq |S_{T}|$.
	Let now $C_i^{T}$ be the set of iteration counters up to $T-1$ corresponding to Case $i$ steps for $i \in \{1,2,3\}$, which satisfies $|C_i^{T}| = n_i(T)$.
	We have by summing \eqref{c1} and \eqref{c2} for the indices in $C_1^{T}$ and $C_2^{T}$ respectively
	\begin{equation} \label{telescopic}
	\sum_{k \in C_1^{T}} f(x_k) - f(x_{k+1}) + \sum_{k \in C_2^{T}} f(x_{k+1}) - f(x_k) \geq 
	\sum_{k \in C_1^{T} }\frac{\rho g_k^2}{2L} + \sum_{k \in C_2^{T} }\frac{1}{2}g_k\ . 
	\end{equation} 
	We now lower bound the right-hand side of \eqref{telescopic} in terms of $g^*_{T}$ as follows:
	\begin{equation*}
	\begin{aligned}  
	& \sum_{k \in C_1^{T} }\frac{\rho g_k^2}{2L} + \sum_{k \in C_2^{T} }\frac{1}{2}g_k \geq |C_1^{T}| \min_{k \in C_1^{T}} \frac{\rho g_k^2}{2L} + |C_2^{T}| \min_{k \in C_2^{T}} \frac{g_k}{2} \geq \\ 
	\geq &	(|C_1^{T}| + |C_2^{T}|) \min \left( \frac{\rho (g^*_T)^2}{2L}, \frac{g^*_T}{2} \right) = 	\left [n_1(T) + n_2(T) \right] \min \left(\rho \frac{(g^*_T)^2}{2L}, \frac{g^*_T}{2} \right) \geq \\ 
	\geq & \frac{T}{2}	\min \left( \frac{\rho (g^*_T)^2}{2L}, \frac{g^*_T}{2} \right)\ .
	\end{aligned} 
	\end{equation*}
	Since the left-hand side of $\eqref{telescopic}$ can clearly be upper bounded by $f(x_0) - f^*$ we have
	\begin{equation*}
	f(x_0) - f^* \geq \frac{T}{2} \min \left(\frac{\rho(g^*_{T})^2}{2L}, \frac{g^*_{T}}{2} \right)\ .
	\end{equation*}
	To finish, if $\frac{T}{2} \min \left( \frac{g^*_{T}}{2}, \frac{\rho(g^*_{T})^2}{2L} \right) =\frac{Tg^*_{T}}{4} $ we then have 
	\begin{equation} \label{g*}
	g^*_{T} \leq \frac{4(f(x_0) - f^*)}{T}
	\end{equation}
	and otherwise
	\begin{equation} \label{g*sqrt}
	g^*_{T} \leq \sqrt{\frac{4L(f(x_0) - f^*)}{\rho T}}\ .
	\end{equation}
	The claim follows by taking the max in the system formed by \eqref{g*} and \eqref{g*sqrt}.
\end{proof}
When the stepsizes coincide with the lower bounds $\bar{\alpha}_k$ or are obtained using exact linesearch, we have the following corollary: 
\begin{Cor} \label{cor:nonconvb}
	Under the assumptions of Theorem \ref{nonconvb}, if $\alpha_k = \bar{\alpha}_k$ or if $\alpha_k$ is selected by exact linesearch then for every $T \in \mathbb{N}$
	\begin{equation} \label{eq:g*rate}
	g_T^* \leq \max\left(\sqrt{\frac{8L (f(x_0) - f^*)}{T}}, \frac{4(f(x_0) - f^*)}{T} \right)\  .	
	\end{equation}
\end{Cor}
\begin{proof}
	By points 2 and 3 of Lemma \ref{alphacond}, relation \eqref{eq:rho} is satisfied with $\rho= \frac{1}{2}$ for both  $\alpha_k= \bar{\alpha}_k$ and $\alpha_k$ given by exact linesearch, and we also have $\alpha_k \geq \bar{\alpha}_k$ in both cases. The conclusion follows directly from Theorem \ref{nonconvb}. 
\end{proof}

\subsection{A general active set identification result}

We can now give a general active set identification result in the nonconvex setting. 
While we won't use strict complementarity when the stepsizes are given by \eqref{alphabound}, without this assumption we will need strict complementarity. Notice that if $A\subseteq \mathcal{X}^*$ enjoys the SIP and if strict complementarity is satisfied for every $x \in A$, then as a direct consequence of \eqref{eq:supp} we have 
\begin{equation} \label{eq:strictcompl} 
	\mathrm{supp}(x) = [1:n] \sm I^c(x) = [1:n]\sm I^c_A
\end{equation}
for every $x \in A$. In this case we can then define $\mathrm{supp}(A)$ as the (common) support of the points in $A$.  
\begin{Th} \label{nonconvid}
	Let $\{x_k\}$ be the sequence generated by the AFW method with stepsizes satisfying $\alpha_k \geq \bar{\alpha}_k$ and \eqref{eq:rho}, where $\bar{\alpha}_k$ is given by  \eqref{alphabound}.
	Let $\mathcal{X}^*$ be the subset of stationary points of $f$. We have: 
	\begin{itemize}
		\item[(a)] $x_k \rightarrow \mathcal{X}^*$.
		\item[(b)] If $\alpha_k= \bar{\alpha}_k$ then $\{x_k\}$ converges to a connected component $A$ of $\mathcal{X}^*$. If additionally $A$ has the SIP then $\{x_k\}$ identifies $I^c_A$ in finite time.
	\end{itemize}
	Assume now that $\mathcal{X}^* = \bigcup_{i = 1}^{C}A_i$ with $\{A_i\}_{i=1}^C$ compact, with distinct supports and such that  $A_i$ has the SIP for each $i\in [1\! : \! C]$.
	\begin{itemize}
		\item[(c)] If ${\alpha}_k \geq \bar\alpha_k$ and if strict complementarity holds for all points in $\mathcal{X}^*$ then $\{x_k\}$ converges to $A_l$ for some $l \in [1:C]$ and identifies $I^c_{A_{l}}$ in finite time.
    \end{itemize}	
\end{Th}
\begin{proof}
	a) By the proof of Theorem \ref{nonconvb} and the continuity of the multiplier function we have 
	\begin{equation} \label{eq10:*}
	x_{k(j)} \rightarrow g^{-1}(0) = \mathcal{X}^*\ ,
	\end{equation}
	where $\{k(j)\}$ is the sequence of indexes corresponding to Case 1 or Case 2 steps. Let $k'(j)$ be the sequence of indexes corresponding to Case 3 steps. Since for such steps $\alpha_{k'(j)} = \bar{\alpha}_{k'(j)}$  we can apply Corollary \ref{nonconv0} to obtain
	\begin{equation} \label{eq10:'}
	\n{x_{k'(j)} - x_{k'(j)+1}} \rightarrow 0\ .	
	\end{equation}
	Combining \eqref{eq10:*}, \eqref{eq10:'} and the fact that there can be at most $n-1$ consecutive Case 3 steps, we get $x_k \rightarrow \mathcal{X}^*$. \\
	b)  By the boundedness of $f$ and point 2 of Lemma \ref{alphacond} if $\alpha_k= \bar{\alpha}_k$ then $\n{x_{k+1} - x_k} \rightarrow 0$. It is a basic topology fact that if $\{x_k\}$ is bounded and $\n{x_{k+1} - x_k} \rightarrow 0$ then the set of limit points of $\{x_k\}$ is connected. This together with point a) ensures that the set of limit points must be contained in a connected component $A$ of $\mathcal{X}^*$. By Theorem \ref{activecompl} it follows that if $A$ has constant support $\{x_k\}$ identifies $I^c_A$ in finite time. \\
	c) Consider a disjoint family of subsets $\{U_i\}_{i=1}^C$ of $\Delta_{n-1}$ with $U_i = \{x \in \Delta_{n-1} \ | \ \dist_1(x, A_i) \leq r_i \}$ where $r_i$ is small enough to ensure some conditions that we now specify. First, we need 
	\begin{equation*}
	   r_i < \frac{\delta_{\min}(A_i)}{2L + \delta_{\min}(A_i)} 
	\end{equation*}
 so that $r_i$ is smaller than the active set radius of every $x \in A_i$ and in particular for every $x \in U_i$ there exists $x^* \in A_i$ such that 
 \begin{equation} \label{eq:xinU}
  \n{x-x^*}_1 < \frac{\delta_{\min}(x^*)}{2L + \delta_{\min}(x^*)}.   
 \end{equation}
Second, we choose $r_i$ small enough so that $\{U_i\}_{i=1}^C$ are disjoint and
	\begin{equation} \label{suppincl}
	\supp(y) \supseteq \supp(A_i) \ \forall y \in U_i\ ,
	\end{equation}
	where these conditions can be always satisfied thanks to the compactness of $A_i$. \\
	Assume now by contradiction that the set $S$ of limit points of $\{x_k\}$ intersects more than one of the $\{A_i\}_{i=1}^C$. Let in particular $A_{l}$ minimize $|\mathrm{supp}(A_l)|$ among the sets containing points of $S$. By point a) $x_k \in \cup_{i=1}^C U_i$ for $k\geq M$ large enough and we can define an infinite sequence  $\{t(j)\}$ of exit times greater than $M$ for $U_{l}$ so that $x_{t(j)} \in U_{l}$ and $x_{t(j) + 1} \in \cup_{i\in [1:C] \sm l} U_i $. Up to considering a subsequence we can assume $x_{t(j) + 1} \in U_{m}$ for a fixed $m\neq l$ for every $j \in \mathbb{N}_0$. \\
	We now distinguish two cases as in the proof of Theorem \ref{ascmain}, where notice that by equation \eqref{eq:xinU} the hypotheses of Theorem \ref{ascmain} are satisfied for $k=t(j)$ and some $x^* \in A_l$. \\
	\textbf{Case 1.} $(x_{t(j)})_h = 0$ for every $h \in I^c_{A_l}$. In the notation of Theorem \ref{ascmain} this corresponds to the case $|J_{t(j)}| = 0$. Then by \eqref{lambdaik} we also have $\lambda_h(x_{t(j)}) > 0$ for every $h \in I^c_{A_l}$. Thus $(x_{t(j)+1})_h = (x_{t(j)})_h = 0$ for every $h \in I^c_{A_l}$ by Lemma \ref{awstep}, so that we can write 
	\begin{equation} \label{incleq}
	\supp(A_m) \subseteq \supp(x_{t(j) + 1}) \subseteq [1:n] \sm I^c_{A_l} 	= \mathrm{supp}(A_l),
	\end{equation}
	where the first inclusion is justified by \eqref{suppincl} for $i=m$ and the second by strict complementarity (see also \eqref{eq:strictcompl} and the related discussion).  
    But since by hypothesis $\supp(A_m) \neq \supp({A_l})$ the inclusion \eqref{incleq} is strict and so it is in contradiction with the minimality of $|\mathrm{supp}(A_l)|$. \\
    \textbf{Case 2.} |$J_{t(j)}| > 0$. Then reasoning as in the proof of Theorem \ref{ascmain} we obtain $d_{t(j)} = x_{t(j)} - e_{\bar{h}}$ for some $\bar{h} \in J_{t(j)} \subset I^c_{A_l}$. Let $\tilde{x}^* \in A_{l}$, and let $\tilde{d} = \alpha_{t(j)} d_{t(j)}$. The sum of the components of $\tilde{d}$ is $0$ with the only negative component being $\tilde{d}_{\bar{h}}$ and therefore 
    \begin{equation} \label{eq:djh}
\tilde{d}_{\bar{h}} = - \sum_{h \in [1:n] \sm \bar{h}} \tilde{d}_h =  - \sum_{h \in [1:n] \sm \bar{h}} |\tilde{d}_h| 
    \end{equation}
    We claim that $\n{x_{t(j) + 1} - \tilde{x}^*}_1 \leq \n{x_{t(j)} - \tilde{x}^*}_1$. This is enough to finish because since $\tilde{x}^*\in A_{l}$ is arbitrary then it follows $\dist_1(x_{t(j)+1}, A_{l}) \leq \dist_1(x_{t(j)}, A_{ l})$ so that $x_{t(j) + 1} \in U_{l}$, a contradiction. \\
    We have
    \begin{equation*}
    \begin{aligned}
       & \n{\tilde{x}^* - x_{t(j) + 1}}_1 = \n{\tilde{x}^* - x_{t(j)} -\alpha_{t(j)}d_{t(j)} }_1 =  \\ 
       = &|\tilde{x}^*_{\bar{h}} - (x_{t(j)})_{\bar{h}} -\tilde{d}_{\bar{h}}| + \sum_{h \in [1:n] \sm \bar{h}} |\tilde{x}^*_h - (x_{t(j)})_h -\tilde{d}_h| = \\   
       = & |\tilde{x}^*_{\bar{h}} - (x_{t(j)})_{\bar{h}}| + \tilde{d}_{\bar{h}} + \sum_{h\in [1:n] \sm \bar{h}} |\tilde{x}^*_h - (x_{t(j)})_h -\tilde{d}_h| \leq \\ 
        \leq & |\tilde{x}^*_{\bar{h}} - (x_{t(j)})_{\bar{h}}| + \tilde{d}_{\bar{h}} + \sum_{h \in [1:n] \sm \tilde{h}} (|\tilde{x}^*_h - (x_{t(j)})_h | + |\tilde{d}_h|) = \\
         = & \n{x_{t(j)} - \tilde{x}^*}_1 + \tilde{d}_{\bar{h}} + \sum_{h \in [1:n] \sm \bar{h}} |\tilde{d}_h| = \n{x_{t(j)} - \tilde{x}^*}_1  
    \end{aligned}
    \end{equation*}
    where in the third equality we used $0=\tilde{x}^*_{\bar{h}} \leq - \tilde{d}_{\bar{h}} \leq (x_{t(j)})_{\bar{h}}$ and in the last equality we used \eqref{eq:djh}. \\
    Reasoning by contradiction we have proved that all the limit points of $\{x_k\}$ are in $A_{l}$ for some $l~\in~[1,...,C]$. The conclusion follows immediately from Theorem \ref{activecompl}. 
\end{proof}

\subsection{Quantitative version of active set identification} 
We now assume that the gap function $g(x)$ satisfies the H\"olderian error bound condition 
\begin{equation} \label{hbg}
g(x) \geq \theta \dist_1(x, \mathcal{X}^*)^p
\end{equation}
for some $\theta, p > 0$ (see e.g. \cite{bolte2017error} for some example). This is true for instance if the components of $\nabla f(x)$ are semialgebraic functions. We have the following active set complexity bound:
\newcommand{\varepsilonb}{\bar{\varepsilon}}
\begin{Th} \label{nonconvcompl}
	Assume $\mathcal{X}^* = \bigcup_{i \in [1:C]} A_i$ where $A_i$ is compact and with the SIP for every $i \in [1:C]$ and $0< d \myeq \min_{\{i,j\} \subset [1:C]} \dist_1(A_i, A_j)$. Let $r_*$ be the minimum active set radius of the sets $\{A_i\}_{i=1}^C$. Let $q(\varepsilon): \R_{>0} \rightarrow \mathbb{N}_0$ be such that 
	$f(x_k)-f(x_{k+1}) \leq \varepsilon$ for every $k \geq q(\varepsilon)$, and assume that $g(x)$ satisfies \eqref{hbg}. Assume that the stepsizes satisfy $\alpha_k = \bar{\alpha}_k$, with $\bar{\alpha}_k$ given by \eqref{alphabound}. Then the active set complexity is at most $q(\bar{\varepsilon}) + 2n$ for $\bar{\varepsilon}$ satisfying the following conditions
	\begin{equation} \label{e1:epscond}
	\begin{aligned}
	\varepsilonb & < L\, ,  \quad \left(\frac{2\sqrt{L\varepsilonb}}{\theta}\right)^{\frac{1}{p}} < r_*\, \quad\mbox{and }\quad 
	2\left(\frac{2\sqrt{L\varepsilonb}}{\theta}\right)^{\frac{1}{p}} + 2n \sqrt{\frac{2\varepsilonb}{L}} & \leq d\ .	
	\end{aligned} 		
	\end{equation}
\end{Th} 
The proof is substantially a quantitative version of the argument used to prove point b) of Theorem \ref{nonconvid}. 
\begin{proof}
	Fix $k \geq q(\bar{\varepsilon})$, so that 
	\begin{equation} \label{e1:eps}
	f(x_{k}) - f(x_{k+1}) \leq \varepsilonb\ .
	\end{equation}
	We will refer to Case $i$ steps for $i \in [1:3]$ following the definitions in Theorem \ref{nonconvb}. If the step $k$ is a Case 1 step, then by \eqref{c1} with $\rho= 1/2$ we have 
	\begin{equation*}
	f(x_{k}) - f(x_{k+1}) \geq \frac{g(x_k)^2}{4L}
	\end{equation*}
	and this together with \eqref{e1:eps} implies	
	\begin{equation*}
	2\sqrt{L\bar{\varepsilon}} \geq 2\sqrt{L(f(x_k)-f(x_{k+1}))} \geq g(x_k)\ .
	\end{equation*}	
	Analogously, if the step $k$ is a Case 2 step, then by \eqref{c2} we have
	\begin{equation*}
	f(x_{k}) - f(x_{k+1}) \geq \frac{g(x_k)}{2}
	\end{equation*}
	so that $2\bar{\varepsilon} \geq g(x_k)$. By the leftmost condition in~\eqref{e1:epscond} we have $\varepsilonb < L$ so that $2\sqrt{L\bar{\varepsilon}} \geq 2\bar{\varepsilon}$, and therefore for both Case 1 and Case 2 steps we have 
	\begin{equation} \label{e1:gkbound}
	g(x_k) \leq 2\sqrt{L\bar{\varepsilon}}\ .
	\end{equation} 
	By inverting relation \eqref{eq:lim}, we also have 
	\begin{equation} \label{e1:nbound}
	\n{x_k - x_{k+1}} \leq \sqrt{\frac{2(f(x_k) - f(x_{k+1}))}{L}} \leq \sqrt{\frac{2\varepsilonb}{L}}\ .
	\end{equation}
	Now let $\bar{k} \geq q(\varepsilonb)$ be such that step $\bar{k}$ is a Case 1 or Case 2 step. By the error bound condition together with \eqref{e1:gkbound} 
	\begin{equation} \label{e1:d1bound}
	\dist_1(x_{\bar{k}}, \mathcal{X}^*) \leq \left(\frac{g(x_{\bar{k}})}{\theta}\right)^{\frac{1}{p}} \leq  \left( \frac{2\sqrt{L\bar{\varepsilon}}}{\theta}\right)^{\frac{1}{p}} < r_*\ ,
	\end{equation}
	where we used \eqref{e1:gkbound} in the second inequality and the second condition of \eqref{e1:epscond} in the third inequality.  
	In particular there exists $l$ such that $\dist_1(x_{\bar{k}}, A_{l}) \leq (2\sqrt{L\bar{\varepsilon}}/\theta)^{1/p}$. We claim now that $I^c_{A_{l}}$ is identified at latest at step $\bar{k} + n$. \\
	First, we claim that for every Case 1 or Case 2 step with index $\tau \geq \bar{k}$ we have $\dist_1(x_{\tau}, A_{l})\leq (g(x_{\tau})/\theta)^{1/p} $. We reason by induction on the sequence $\{s(k')\}$ of Case 1 or Case 2 steps following $\bar{k}$, so that in particular $s(1) = \bar{k}$ and $\dist_1(x_{s(1)}, A_{l})\leq g(x_{s(1)}) $ is true by \eqref{e1:d1bound}. Since there can be at most $n-1$ consecutive Case 3 steps, we have $s(k'+ 1) - s(k') \leq n$ for every $k' \in \mathbb{N}_0$. Therefore
	\begin{equation} \label{e1:0p}
	\begin{aligned}
	\n{x_{s(k')} - x_{s(k'+1)}}_1 \leq & \sum_{i=s(k')}^{s(k'+1)-1}\n{x_{i+1}-x_i}_1 \leq 2\sum_{i=s(k')}^{s(k'+1)-1}\n{x_{i+1}-x_i} \leq \\ \leq & 2[s(k'+1)-s(k')]\sqrt{\frac{2\varepsilonb}{L}} \leq 2n\sqrt{\frac{2\varepsilonb}{L}}\ ,
	\end{aligned}
	\end{equation}  
	where in the second inequality we used part 3 of Lemma \ref{simpelem} to bound each of the summands of the left-hand side, and in the third inequality we used  \eqref{e1:nbound}. Assume now by contradiction $\dist_1(x_{s(k'+1)}, A_{l}) > (g(x_{s(k' + 1)})/\theta)^{1/p} $. Then by \eqref{e1:d1bound} applied to $s(k'+1)$ instead of $\bar{k}$ there must exists necessarily $j\neq l$ such that $\dist_1(x_{s(k'+1)}, A_{j}) \leq (g(x_{s(k' + 1)})/\theta)^{1/p}$. In particular we have 
	\begin{equation} \label{e1:2p}
	\begin{aligned}
	\n{x_{s(k')} - x_{s(k'+1)}}_1 \geq & \dist_1(A_{l}, A_j) - \dist_1(x_{s(k'+1)}, A_{j}) - \dist_1(x_{s(k')}, A_{l}) \geq \\ \geq & d - \left( \frac{g(x_{s(k')})}{\theta} \right)^{\frac{1}{p}} - \left( \frac{g(x_{s(k'+1)})}{\theta} \right)^{\frac{1}{p}} \geq d-2 \left( \frac{2\sqrt{L\bar{\varepsilon}}}{\theta}\right)^{\frac{1}{p}}\ ,
	\end{aligned}
	\end{equation}  
	where we used~\eqref{e1:gkbound} in the last inequality. 
	But by the second condition of \eqref{e1:epscond}, we have  
	\begin{equation} \label{e1:1p}
	d- 2 \left( \frac{2\sqrt{L\bar{\varepsilon}}}{\theta}\right)^{\frac{1}{p}} > 2n\sqrt{\frac{2\varepsilonb}{L}}\ .
	\end{equation}
	Concatenating \eqref{e1:0p}, \eqref{e1:1p} and \eqref{e1:2p} we get a contradiction and the claim is proved. Notice that an immediate consequence of this claim is $\dist_1(x_{\tau}, A_{l}) < r_*$ by \eqref{e1:d1bound} applied to $\tau$ instead of $\bar{k}$, where $\tau \geq \bar{k}$ is an index corresponding to a Case 1 or Case 2 step. \\
	To finish the proof, first notice that there exists an index $\bar{k} \in [q(\varepsilonb), q(\varepsilonb) + n]$ corresponding to a Case 1 or Case 2 step, since there can be at most $n-1$ consecutive Case 3 steps. Furthermore, since by \eqref{e1:d1bound} we have $\dist_1(x_{\bar{k}}, A_{l}) < r_*$, by the local identification Theorem \ref{ascmain} in the steps immediately after $\bar{k}$ the AFW identifies one at a time the variables in $I^c_{A_{l}}$, so that there exists $h \leq n$ such that $(x_{\bar{k} + h})_i=0$ for every $i \in I^c_{A_{l}}$. Moreover, by the claim every Case 1 and Case 2 step following step $\bar{k}$ happens for points inside $B_1(A_{l}, r_*)$ so it does not change the components corresponding to $I^c_{A_{l}}$ by the local identification Theorem \ref{ascmain}. At the same time, Case 3 steps do not increase the support, so that $(x_{\bar{k} + h + l})_i=0$ for every $i \in I^c_{A_{l}}$, $l \geq 0$. Thus active set identification happens in $\bar{k} + h \leq q(\varepsilonb) + n + h \leq q(\varepsilonb) + 2n$ steps.  
\end{proof}
\begin{Rem}
	Assume that the set of stationary points is finite, so that $A_i = \{a_i\}$ for every $i \in [1\!:\!C]$ with $a_i \in \Delta_{n-1}$. Let \begin{equation}
	c_{\min} = \min_{i\in [1:C] } \min_{j:(a_i)_j \neq 0} (a_i)_j
	\end{equation}
	be the minimal nonzero component of a stationary point. Then one can prove a $q(\varepsilonb) + n$ active set identification bound replacing \eqref{e1:epscond} with the following condition on $\bar{\varepsilon}$ which has no explicit dependence on $n$:
	\begin{equation*}
	\begin{aligned}
	\varepsilonb <& L, \quad r(\varepsilonb) + l(\varepsilonb) < \min(r_*, d/2, c_{\min}/2 )\ ,
	\end{aligned}
	\end{equation*}	
	where $r(\varepsilonb) = \left(\frac{2\sqrt{L\varepsilonb}}{\theta}\right)^{\frac{1}{p}}$ and $l(\varepsilonb) = 2\sqrt{\frac{2\varepsilonb}{L}}$. We do not discuss the proof since it roughly follows  the same lines of Theorem \ref{nonconvcompl}'s proof. 
\end{Rem}

\begin{Rem}
When we have an explicit expression for the convergence rate $q(\varepsilon)$, then we can get an active set complexity bound using Theorem \ref{nonconvcompl}.
\end{Rem}

\subsection{Local active set complexity bound}
A key element to ensure local convergence to a strict local minimum will be the following property
\begin{equation}\label{strongdescent}
x_{k} \in \textnormal{argmax}\{f(x) \ | \ x \in \textnormal{conv}(x_k, x_{k+1}) \}\ .
\end{equation}
which in particular holds when $\alpha_k = \bar{\alpha}_k$ as it is proved in Lemma \ref{alphacond}. 
The property \eqref{strongdescent} is obviously stronger than the usual monotonicity, and it ensures that the sequence cannot escape from connected components of sublevel sets. When $f$ is convex it is immediate to check that \eqref{strongdescent} holds if and only if $\{f(x_k)\}$ is monotone non increasing. 
\\ 
Let $x^*$ be a  stationary point which is also a strict local minimizer  isolated from the other stationary points and $\tilde{f} = f(x^*)$.
Let then $\beta$ be such that there exists a connected component $V_{x^*, \beta}$ of $f^{-1}((-\infty, \beta])$ satisfying 
\begin{equation*}
  V_{x^*, \beta} \cap \mathcal{X}^* = \{x^*\} = \argmin_{x \in V_{x^*, \beta}} f(x). 
\end{equation*}  

\begin{Th} \label{nonconv}
	Let $\{x_k\}$ be a sequence generated by the AFW, with $x_0 \in V_{x^*, \beta}$ and with stepsize given by \eqref{alphabound}.
	Let 
	\begin{equation*}
	r_* = \frac{\delta_{\min}(x^*)}{2L + \delta_{\min}(x^*)}\ .
	\end{equation*}
	Then $x_k \rightarrow x^*$ and the sequence identifies the support in at most
	\begin{equation*}
	\left\lceil\max \left(\frac{4(f(x_0)-\tilde{f})}{\tau}, \frac{8L(f(x_0)-\tilde{f})}{\tau^2} \right)\right\rceil + 1 + |I^c(x^*)|
	\end{equation*}
	steps with 
	\begin{equation*}
	\tau = \min\{g(x) \ | \ x \in f^{-1}([m, + \infty)) \cap  V_{x^*, \beta} \}\ ,
	\end{equation*} 
	where 
	$$ m = \min \{\ f(x) \ | \  x \in V_{x^*, \beta} \sm B_{r_*}(x^*) \}\ . $$
\end{Th}  
\begin{proof}
	We have all the hypotheses to apply the bound given in Corollary \ref{cor:nonconvb} for $g^*_k$:
	\begin{equation*} 
	g^*_k \leq \max \left(\sqrt{\frac{8L(f(x_0) - f^*)}{k}}, \frac{4(f(x_0) - f^*)}{k}
	\right)\ .
	\end{equation*}
	It is straightforward to check that if
	\begin{equation*}
	\bar{h} =  \left\lceil \max\left(\frac{4(f(x_0) - f^*)}{\tau}, \frac{8L(f(x_0)-f^*)}{\tau^2} \right) \right\rceil+1
	\end{equation*}
	then
	\begin{equation*}
	g^*_{\bar{h}} < \tau\ .
	\end{equation*}
	Therefore, by the definition of $\tau$ ,we get $f(x_{\bar{h}}) < m$. 
	We claim that $x_h \in B_{r_*}(x^*)$ for every $h \geq \bar{h}$. Indeed by point 1 of Lemma \ref{alphacond} the condition $\alpha_k = \bar{\alpha}_k$ on the stepsizes imply that $\{x_k\}$ satisfies \eqref{strongdescent} and it can not leave connected components of level sets. Thus since $f(x_h) < m$ we have 
	$$ x_h \in V_{x^*, \beta} \cap f^{-1}(-\infty, m) \subset B_{r_*}(x^*)\ ,$$ 
	where the inclusion follows directly from the definition of $m$. We can then apply the local active set identification Theorem \ref{ascmain} to obtain an active set complexity of 
	\begin{equation*}
	\bar{h} + |I^c(x^*)| =  \left\lceil\max\left(\frac{4(f(x_0) - f^*)}{\tau}, \frac{8L(f(x_0)-f^*)}{\tau^2} \right) \right\rceil + 1 +  |I^c(x^*)|\ , 
	\end{equation*}
thus getting our result.	
\end{proof}

\section{Conclusions}\label{concl}
We proved general results for the AFW finite time active set convergence problem, giving explicit bounds on the number of steps necessary to identify the support of a solution. As applications of these results we computed the active set complexity for strongly convex functions and nonconvex functions. Possible expansions of these results would be to adapt them for other FW variants and, more generally, to other first order methods. It also remains to be seen if these identification properties of the AFW can be extended to problems with nonlinear constraints.    
\section{Appendix} \label{tech}
In several proofs we need some elementary inequalities concerning the euclidean norm $\n{\cdot}$ and the norm $\n{\cdot}_1$. 
\begin{Lemma} \label{simpelem}
	Given $\{x, y\}\subset \Delta_{n-1}$, $i \in [1:n]$:
	\begin{itemize}
		\item[1.] $\|e_i - x\| \leq \sqrt{2}(e_i - x)_i$;
		\item[2.] $(y-x)_i \leq \|y - x\|_1/2 $
		\item[3.] If $\{x_k\}$ is a sequence generated on the probability simplex by the AFW then $\n{x_{k+1}-x_k}_1 \leq 2\n{x_{k+1}-x_k}$ for every $k$. 
	\end{itemize}
\end{Lemma}
\begin{proof}
	1. $(e_i - x)_j = -x_j$ for $j \neq i$, $(e_i - x)_i = 1-x_i = \sum_{j \neq i}x_j$. In particular
	\begin{equation*}
	\begin{aligned}
	\|e_i - x\| = (\sum_{j\neq i} x_j^2 + (e_i-x)_i^2)^{\frac{1}{2}} \leq ((\sum_{j\neq i} x_j)^2+ (1-x_i)^2)^{\frac{1}{2}} = \sqrt{2}(\sum_{j \neq i} x_j ) = \sqrt{2}(e_i - x)_i 
	\end{aligned}
	\end{equation*}
	2. Since $\sum_{j \in [1:n]} x_j = \sum_{j \in [1:n]}y_j$ so that $\sum (x-y)_j = 0$ we have 
	$$ (y-x)_i = \sum_{j \neq i} (x-y)_j $$
	and as a consequence
	\begin{equation*}
	\|y-x\|_1 = \sum_{j \in [1:n]} |(y-x)_j| \geq (y-x)_i + \sum_{j \neq i} (x-y)_j = 2(y-x)_i\ .
	\end{equation*}
	3. We have $x_{k+1} - x_k = \alpha_k d_k$ with $d_k = \pm (e_i - x_k)$ for some $i \in [1:n]$.  By homogeneity it suffices to prove $\n{d_k} \geq \frac{1}{2}\n{d_k}_1$. We have
	\begin{equation*}
		\n{d_k} \geq 1-(x_k)_i = \frac{1}{2} (1-(x_k)_i + \sum_{j \neq i} (x_k)_j) = \frac{1}{2}\n{d_k}_1\ ,
	\end{equation*}
	where in the first equality we used $\sum_{i=1}^n (x_k)_i=1$ and in the second equality we used $0\leq x_k \leq 1$.
\end{proof}

\subsection{Technical results related to stepsizes} \label{nonstationary}
We now prove several properties related to the stepsize given in \eqref{alphabound}.
\begin{Lemma}\label{alphacond}
	Consider a sequence $\{x_k\}$ in $\Delta_{n - 1}$ such that $x_{k+1} = x_k + \alpha_k d_k$ with $\alpha_k \in \R_{\geq 0}$, $d_k \in \R^n$. Let $\bar{\alpha}_k$ be defined as in \eqref{alphabound}, let $p_k = -\nabla f(x_k)^\top d_k$ and assume $p_k > 0$. Then: 
	\begin{itemize}
		\item[1.] If $0 \leq \alpha_k \leq 2p_k /(\|d_k\|^2L)$, the sequence $\{x_k\}$ has the property \eqref{strongdescent}.
		\item[2.] If $\alpha_k = \bar{\alpha}_k$ then \eqref{eq:rho} is satisfied with $\rho = \frac{1}{2}$. Additionally, we have
		\begin{equation} \label{eq:lim}
		f(x_k) - f(x_{k+1}) \geq L\frac{\n{x_{k+1} - x_k}^2}{2}\ .
		\end{equation}
		\item[3.] If $\alpha_k$ is given by exact linesearch, then $\alpha_k \geq \bar{\alpha}_k$ and \eqref{eq:rho} is again satisfied with $\rho= \frac{1}{2}$. 
	\end{itemize}
	
\end{Lemma}
\begin{proof}
	By the standard descent lemma~\cite[Proposition 6.1.2]{bertsekas2015convex} we have
	\begin{equation} \label{e11:std}
	f(x_k) - f(x_k + \alpha d_k) \geq \alpha p_k - \alpha^2 \frac{L\|d_k\|^2}{2}\ .
	\end{equation}
	It is immediate to check 
	\begin{equation} \label{eq11:ds}
	\alpha \nabla f(x_k)^\top d_k +\alpha^2 \frac{L\|d_k\|^2}{2} \leq 0\ ,
	\end{equation}
	for every $0 \leq \alpha \leq \frac{2p_k}{L\|d_k\|^2} $ and 
	\begin{equation} \label{eq11:lin}
	\alpha p_k - \alpha^2 \frac{L\|d_k\|^2}{2} \geq \alpha p_k/2 \geq \alpha^2 \frac{L\|d_k\|^2}{2}
	\end{equation}
	for every $0 \leq \alpha \leq \frac{p_k}{L\|d_k\|^2} $. \\
	1. For every $x \in \textnormal{conv}(x_k, x_{k+1})\subseteq \left\{x_k+ \alpha d_k \ | \ 0 \leq \alpha \leq  \frac{2p_k}{L\|d_k\|^2} \right \}$, we have
	\begin{equation*}
	f(x) = f(x_k + \alpha d_k) \leq f(x_k) + \alpha \nabla f(x_k)^\top d_k +\alpha^2 \frac{L\|d_k\|^2}{2} \leq f(x_k)\ ,
	\end{equation*}
	where we used \eqref{e11:std} in the first inequality and \eqref{eq11:ds} in the second inequality. \\
	2. We have 
	\begin{equation*}
	f(x_k) - f(x_{k+1}) = f(x_k) - f(x_k + \bar{\alpha}_k d_k) \geq \bar{\alpha}_k p_k/2\ ,
	\end{equation*} 
	where we have the hypotheses to apply \eqref{eq11:lin} since $0 \leq \bar{\alpha}_k \leq \frac{p_k}{L\|d_k\|^2} $. 
	Again by \eqref{eq11:lin}
	\begin{equation*}
	f(x_k) - f(x_{k+1}) = f(x_k) - f(x_k + \bar{\alpha}_k d_k) \geq \bar{\alpha}_k^2 \frac{L\|d_k\|^2}{2} = L \frac{\n{x_k - x_{k+1}}^2}{2}\ .
	\end{equation*}
	3. If $\alpha_k = \alpha_k^{\max}$ then there is nothing to prove since $\bar{\alpha}_k \leq \alpha_k^{\max}$. Otherwise we have 
	\begin{equation} \label{eq11:linesearch}
	0 = \frac{\partial}{\partial \alpha} f(x_k + \alpha d_k) |_{\alpha= \alpha_k} =d_k^\top (\nabla f(x_k + \alpha_k d_k))
	\end{equation}
	and therefore  
	\begin{equation} \label{eq11:linesearch2}
	\begin{aligned}
	- d_k^\top \nabla f(x_k) & = - d_k^\top \nabla f(x_k) +	d_k^\top \nabla f(x_k + \alpha_k d_k) = - d_k^\top (\nabla f(x_k) - \nabla f(x_k + \alpha_k d_k)) \\ 
	& \leq L\n{d_k}\n{x_k - (x_k + \alpha_k d_k) } =  \alpha_k L\n{d_k}^2\ ,
	\end{aligned}	
	\end{equation}
	where we used \eqref{eq11:linesearch} in the first equality and the Lipschitz condition in the inequality. From \eqref{eq11:linesearch2} it follows 
	\begin{equation*}
	\alpha_k \geq \frac{- d_k^\top \nabla f(x_k)}{L \n{d_k}^2} \geq \bar{\alpha}_k
	\end{equation*}
	and this proves the first claim. As for the second,
	\begin{equation*}
	f(x_k) - f(x_k + \alpha_k d_k) \geq f(x_k) - f(x_k +  \bar{\alpha}_kd_k) \geq \frac{\bar\alpha_k}{2}p_k\ , 
	\end{equation*}
	where the first inequality follows from the definition of exact linesearch and the second by point 2 of the lemma. 
\end{proof}
\begin{Cor}\label{nonconv0}
	Under the hypotheses of Lemma \ref{alphacond}, assume that $f(x_k)$ is monotonically decreasing and assume that for some subsequence $k(j)$ we have $x_{k(j) + 1} = x_{k(j)} + \bar{\alpha}_{k(j)} d_{k(j)}$. Then $$\n{x_{k(j)} - x_{k(j)+1}} \rightarrow 0\ .$$ 	
\end{Cor}
\begin{proof}
	By~\eqref{eq:lim} we have  
	\begin{equation*}
	f(x_{k(j)}) - f(x_{k(j) + 1}) \geq \frac{L}{2}\n{x_{k(j)} - x_{k(j)+1}}^2 
	\end{equation*}
	and the conclusion follows by monotonicity and boundedness. 
\end{proof}
\subsection{AFW complexity for generic polytopes} \label{generalafw}
It is well known as anticipated in the introduction that every application of the AFW to a polytope can be seen as an application of the AFW to the probability simplex. \\
In this section we show the connection between the active set and the face of the polytope exposed by $-\nabla f(y^*)$, where $y^*$ is a stationary point for $f$. We then proceed to show with a couple of examples how the results proved for the probability simplex can be adapted to general polytopes. In particular we will generalize Theorem \ref{activecompl}, thus proving that under a convergence assumption the AFW identifies the face exposed by the gradients of some stationary points. An analogous result is already well known for the gradient projection algorithm, and was first proved in \cite{burke1994exposing} building on \cite{burke1988identification} which used an additional strict complementarity assumption but worked in a more general setting than polytopes, that of convex compact sets with a polyhedral optimal face.  \\
Before stating the generalized theorem we need to introduce additional notation and prove a few properties mostly concerning the generalization of the simplex multiplier function $\lambda$ to polytopes.\\  
Let $P$ be a polytope and $f: P \rightarrow \mathbb{R}^n$ be a function with gradient having Lipschitz constant $L$. \\
To define the AFW algorithm we need a finite set of atoms $\mathcal{A}$ such that $\textnormal{conv}(\mathcal{A}) = P$. As for the probability simplex we can then define for every $a \in \mathcal{A}$ the multiplier function $\lambda_a: P \rightarrow \mathbb{R}$ by
$$\lambda_a(y) = \nabla f(y)^{\top} (a-y)\ . $$
Let finally $A$ be a matrix having as columns the atoms in $\mathcal{A}$, so that $A$ is also a linear transformation mapping $\Delta_{|\mathcal{A}| - 1}$ in $P$ with $Ae_i = A^i \in \mathcal{A}$. \\
In order to apply Theorem \ref{ascmain} we need to check that the transformed problem 
\begin{equation*}
\min \{f(Ax) \ | \ x \in \Delta_{|\mathcal{A}| - 1}\}
\end{equation*}
still has all the necessary properties under the assumptions we made on $f$. \\
Let $\tilde{f}(x) = f(Ax) $.
First, it is easy to see that the gradient of $\tilde{f}$ is still Lipschitz.
Also $\lambda$ is invariant under affine transformation, meaning that $\lambda_{A^i}(Ax) = \lambda_i(x) $ for every $i \in [1:|\mathcal{A}|]$, $x \in \Delta_{|\mathcal{A}|-1}$. Indeed
$$\lambda_{A^i}(Ax) = \nabla f(Ax)^\top (A^i - Ax) = \nabla f(Ax)^\top A(e_i - x) = \nabla \tilde f(x)^\top  (e_i-x) = \lambda_i(x)\ .  $$
Let $Y^*$ be the set of stationary points for $f$ on $P$, so that by invariance of multipliers $\mathcal{X}^* = A^{-1}(Y^*)$ is the set of stationary points for $\tilde{f}$. The invariance of the identification property follows immediately from the invariance of $\lambda$: if the support of the multiplier functions for $f$ restricted to $B$ is $\{A^i\}_{i \in I^c}$, then the support of the multiplier functions for $\tilde{f}$ restricted to $A^{-1}(B)$ is $I^c$. \\
We now show the connection between the face exposed by $-\nabla f$ and the support of the multiplier function. Let $y^*=Ax^* \in Y^*$ and let
\begin{equation*}
P^*(y^*) = \{y \in P \ | \ \nabla f(y^*)^\top y = \nabla f(y^*)^\top y^* \} = \textnormal{argmax}\{-\nabla f(y^*)^\top y \ | \ y \in P \} = \mathcal{F}(-\nabla f(y^*))
\end{equation*}
be the face of the polytope $P$ exposed by $ - \nabla f(y^*)$. The complementarity conditions for the generalized multiplier function $\lambda$ can be stated very simply in 
terms of inclusion in $P^*(y^*)$: since $y^* \in P^*(y^*)$ we have $\lambda_a(y^*) = 0$ for every $a \in P^*(y^*)$, $\lambda_a(y^*) > 0$ for every $a \notin P^*(y^*)$. 
But $P$ is the convex hull of the set of atoms in $\mathcal{A}$ so that the previous relations mean that the face $P^*(y^*)$ is the convex hull of the set of atoms for which $\lambda_a(y^*) = 0$:
\begin{equation*}
P^*(y^*) = \textnormal{conv}\{a \in \mathcal{A} \ | \ \lambda_a(y^*) = 0 \}
\end{equation*}
or in other words since $\lambda_{A^i}(y^*) = 0$ if and only if $i \in I(x^*)=\{i\in [1:n]\ | \ \lambda_i(x^*)=0\}$:
\begin{equation} \label{expconv}
P^*(y^*) = \textnormal{conv} \{a \in \mathcal{A} \ | \  a=A^i, \ i \in I(x^*)\}\ .
\end{equation}
A consequence of \eqref{expconv} is that given any subset $B$ of $P$ with a constant  active set,  we necessarily get $P^*(w)~=~P^*(z)$ for every $w, z \in B$, since $I(w) = I(z)$. For such a subset $B$ we can then define 
\begin{equation*}
P^*(B) = P^*(y^*) \textnormal{ for any }y^* \in B
\end{equation*}
where the definition does not depend on the specific $y^* \in B$ considered.
We can now restate Theorem~\ref{activecompl} in slightly different terms:
\begin{Th} \label{activecompl2}
Let $\{y_k\}$ be a sequence generated by the AFW on $P$ and let $\{x_k\}$ be the corresponding sequence of weights in $\Delta_{|\mathcal{A}|-1}$ such that $\{y_k\} = \{Ax_k\}$. Assume that the stepsizes satisfy $\alpha_k \geq \bar{\alpha}_k$ (using $\tilde{f}$ instead of $f$ in \eqref{alphabound}). If there exists a compact subset $B$ of $Y^*$ with the SIP such that $y_k \rightarrow B$, then there exists $M$ such that 
$$y_k \in P^*(B) \textit{ for every }k\geq M. $$ 
\end{Th}
\begin{proof}
	Follows from Theorem \ref{activecompl} and the affine invariance properties discussed above.
\end{proof}
A technical point concerning Theorem \ref{activecompl2} is that in order to compute $\bar{\alpha}_k $ the Lipschitz constant $L$ of $ \nabla \tilde{f}$ (defined on the simplex) is necessary. When optimizing on a general polytope, the calculation of an accurate estimate of $L$ for $\tilde{f}$ may be problematic. However, by Lemma \ref{alphacond} if the AFW uses exact linesearch, the stepsize  $\bar{\alpha}_k$ (and in particular the constant $L$) is not needed because the inequality $\alpha_k \geq \bar{\alpha}_k$ is automatically satisfied.\\
We now generalize the analysis of the strongly convex case.  
The technical problem here is that strong convexity, which is used in Corollary \ref{ssimplex}, is not maintained by affine transformations, so that instead we will have to use a weaker error bound condition. As a possible alternative, in \cite{lacoste2015global} linear convergence of the AFW is proved with dependence only on affine invariant parameters, so that any version of Theorem \ref{ascmain} and Corollary \ref{ssimplex} depending on those parameters instead of $u_1, L$ would not need this additional analysis. \\ 
Let  $P = \{ y \in \mathbb{R}^n \ | \ Cy \leq b\}$, $y^*$ be the unique minimizer of $f$ on $P$ and $u > 0$ be such that
\begin{equation*}
f(y) \geq f(y^*)+\frac{u}{2} \|y - y^*\|^2\ .
\end{equation*}
The function $\tilde{f}$ inherits the error bound condition necessary for Corollary \ref{ssimplex} from the strong convexity of $f$:
for every $x \in \Delta_{|\mathcal{A}|-1}$ by \cite{beck2017linearly}, Lemma 2.2  we have 
$$\textnormal{dist} (x, \mathcal{X}^*) \leq \theta \|Ax - y^*\|  $$ 
where $\theta$ is the Hoffman constant related to $[C^T, [I; e; -e]^T]^T$. As a consequence if $\tilde{f}^*$ is the minimum of $\tilde{f}$
\begin{equation*}
\tilde{f}(x) - \tilde{f}^* = f(Ax) - f(y^*) \geq \frac{u}{2} \|Ax -y^*\|^2 \geq \frac{u}{2\theta^2} \textnormal{dist}(x, \mathcal{X}^*)^2
\end{equation*}
and using that $n\|\cdot \|^2 \geq \|\cdot \|_1^2$ we can finally retrieve an error bound condition with respect to $\|\cdot \|_1$:
\begin{equation} \label{errorbound}
\tilde{f}(x) - \tilde{f}^* \geq \frac{u}{2n\theta^2}\textnormal{dist}_1(x, \mathcal{X}^*)^2.
\end{equation}
Having proved this error bound condition for $\tilde{f}$ we can now generalize \eqref{awdir}:
\begin{Cor}	
	The sequence $\{y_k\}$ generated by the AFW is in $P^*(y^*)$ for
	$$ k \geq \max\left (0, \frac{\textnormal{ln}(h_0) - \textnormal{ln}(u_P r_*^2/2)}{\textnormal{ln}(1/q)}\right ) +|I^c| $$
	where $q\in(0,1)$, is the constant related to the linear convergence rate $f(y_k) - f(y^*) \leq q^k(f(y_0) - f(y^*))$, $u_P = \frac{u}{2n\theta^2}$, $r_* = \frac{\delta_{\min}}{2L + \delta_{\min}}$ with $ \delta_{\min} = \min \{\lambda_a(y^*) \ | \ \lambda_a(y^*) > 0 \}$. 
\end{Cor}
\begin{proof}
	Let $I = \{i \in [1:|\mathcal{A}|] \ | \ \lambda_{A^i}(y^*) = 0 \}$, $P^* = P^*(y^*)$. 
	Since $P^* = \textnormal{conv}(\mathcal{A} \cap P^*)$ and by \eqref{expconv} $\textnormal{conv}(\mathcal{A} \cap P^*) = \textnormal{conv} \{A^i \ | \ i \in I\}$ 
	the theorem is equivalent to prove that for every $k$ larger than the bound, we have $y_k \in \textnormal{conv} \{A^i \ | \ i \in I\}$. 
	Let $\{x_k\}$ be the sequence generate by the AFW on the probability simplex, so that $y_k=Ax_k$. We need to prove that, for every $k$ larger than the bound, we have
	$$ x_k \in \textnormal{conv }\{e_i \ | \ i \in I\}\ , $$
	or in other words $(x_k)_i = 0$ for every $i \in I^c$. \\
	Reasoning as in Corollary \ref{ssimplex} we get that $\textnormal{ dist}_1(x_k,  \mathcal{X}^*) < r_*$ for every 
	\begin{equation} \label{lowbound}
	k \geq \frac{\textnormal{ln}(h_0) - \textnormal{ln}(u_P r_*^2/2)}{\textnormal{ln}(1/q)}\ .
	\end{equation}
	Let $\bar{k}$ be the minimum index such that \eqref{lowbound} holds. 
	For every $k \geq \bar{k}$ there exists $x^* \in \mathcal{X}^*$ with $\|x_k - x^*\|_1 < r_*$. 
	But $\lambda_i(x) = \lambda_{A^i}(y^*)$ for every $x \in \mathcal{X}^*$ by the invariance of $\lambda$, so that we can apply Theorem \ref{ascmain} with fixed point $x^*$ and obtain that if 
	$J_k = \{i \in I^c \ | \ (x_k)_i> 0\}$ then $J_{k+1} \leq \max (0, J_k - 1)$. The conclusion follows exactly as in Corollary \ref{ssimplex}. 	
\end{proof}
\bibliographystyle{plain}  
\bibliography{afwbib} 

\begin{thebibliography}{10}

\bibitem{bach2013learning}
Francis Bach.
\newblock Learning with submodular functions: A convex optimization
  perspective.
\newblock {\em Foundations and Trends in Machine Learning}, 6(2-3):145--373,
  2013.

\bibitem{balashov2019gradient}
Maxim Balashov, Boris Polyak, and Andrey Tremba.
\newblock Gradient projection and conditional gradient methods for constrained
  nonconvex minimization.
\newblock {\em arXiv preprint arXiv:1906.11580}, 2019.

\bibitem{beck2017linearly}
Amir Beck and Shimrit Shtern.
\newblock Linearly convergent away-step conditional gradient for non-strongly
  convex functions.
\newblock {\em Mathematical Programming}, 164(1-2):1--27, 2017.

\bibitem{bertsekas:1982}
Dimitri~P Bertsekas.
\newblock Projected newton methods for optimization problems with simple
  constraints.
\newblock {\em SIAM J. Control Optim.}, 20(2):221--246, 1982.

\bibitem{bertsekas2015convex}
Dimitri~P Bertsekas.
\newblock {\em Convex optimization algorithms}.
\newblock Athena Scientific, Belmont, 2015.

\bibitem{birgin:2002}
Ernesto~G Birgin and Jos{\'e}~Mario Mart{\'\i}nez.
\newblock Large-scale active-set box-constrained optimization method with
  spectral projected gradients.
\newblock {\em Comput. Optim. Appl.}, 23(1):101--125, 2002.

\bibitem{bolte2017error}
J{\'e}r{\^o}me Bolte, Trong~Phong Nguyen, Juan Peypouquet, and Bruce~W Suter.
\newblock From error bounds to the complexity of first-order descent methods
  for convex functions.
\newblock {\em Mathematical Programming}, 165(2):471--507, 2017.

\bibitem{FOIMP}
Immanuel~M {Bomze}, Francesco {Rinaldi}, and Samuel~Rota {Bul{\' o}}.
\newblock {First-order methods for the impatient: support identification in
  finite time with convergent {F}rank-{W}olfe variants}.
\newblock {\em SIAM Journal on Optimization}, 29(3):2211--2226, 2019.

\bibitem{burke1988identification}
James~V Burke and Jorge~J Mor{\'e}.
\newblock On the identification of active constraints.
\newblock {\em SIAM Journal on Numerical Analysis}, 25(5):1197--1211, 1988.

\bibitem{burke1994exposing}
James~V Burke and Jorge~J Mor{\'e}.
\newblock Exposing constraints.
\newblock {\em SIAM Journal on Optimization}, 4(3):573--595, 1994.

\bibitem{burke1990identification}
Jim Burke.
\newblock On the identification of active constraints {II}: The nonconvex case.
\newblock {\em SIAM Journal on Numerical Analysis}, 27(4):1081--1102, 1990.

\bibitem{canon1968tight}
Michael~D Canon and Clifton~D Cullum.
\newblock A tight upper bound on the rate of convergence of frank-wolfe
  algorithm.
\newblock {\em SIAM Journal on Control}, 6(4):509--516, 1968.

\bibitem{clarkson2010coresets}
Kenneth~L Clarkson.
\newblock Coresets, sparse greedy approximation, and the frank-wolfe algorithm.
\newblock {\em ACM Transactions on Algorithms (TALG)}, 6(4):63, 2010.

\bibitem{2017arXiv170307761C}
Andrea {Cristofari}, Marianna {De Santis}, Stefano {Lucidi}, and Francesco
  {Rinaldi}.
\newblock {An Active-Set Algorithmic Framework for Non-Convex Optimization
  Problems over the Simplex}.
\newblock {\em arXiv e-prints}, page arXiv:1703.07761, March 2017.

\bibitem{cristofari:2018new}
Andrea Cristofari, Marianna De~Santis, Stefano Lucidi, and Francesco Rinaldi.
\newblock An active-set algorithmic framework for non-convex optimization
  problems over the simplex.
\newblock {\em arXiv preprint arXiv:1703.07761v2}, 2018.

\bibitem{desantis:2012}
Marianna De~Santis, Gianni Di~Pillo, and Stefano Lucidi.
\newblock An active set feasible method for large-scale minimization problems
  with bound constraints.
\newblock {\em Computational Optimization and Applications}, 53(2):395--423,
  2012.

\bibitem{frank1956algorithm}
Marguerite Frank and Philip Wolfe.
\newblock An algorithm for quadratic programming.
\newblock {\em Naval research logistics quarterly}, 3(1-2):95--110, 1956.

\bibitem{grigas2019stochastic}
Paul Grigas, Alfonso Lobos, and Nathan Vermeersch.
\newblock Stochastic in-face frank-wolfe methods for non-convex optimization
  and sparse neural network training.
\newblock {\em arXiv preprint arXiv:1906.03580}, 2019.

\bibitem{guelat1986some}
Jacques Guelat and Patrice Marcotte.
\newblock Some comments on wolfe's 'away step'.
\newblock {\em Mathematical Programming}, 35(1):110--119, 1986.

\bibitem{hager2011gradient}
William~W Hager, Dzung~T Phan, and Hongchao Zhang.
\newblock Gradient-based methods for sparse recovery.
\newblock {\em SIAM Journal on Imaging Sciences}, 4(1):146--165, 2011.

\bibitem{hager:2006}
William~W Hager and Hongchao Zhang.
\newblock A new active set algorithm for box constrained optimization.
\newblock {\em SIAM J. Optim.}, 17(2):526--557, 2006.

\bibitem{hager2016active}
William~W Hager and Hongchao Zhang.
\newblock An active set algorithm for nonlinear optimization with polyhedral
  constraints.
\newblock {\em Science China Mathematics}, 59(8):1525--1542, 2016.

\bibitem{hare2004identifying}
Warren~L Hare and Adrian~S Lewis.
\newblock Identifying active constraints via partial smoothness and
  prox-regularity.
\newblock {\em Journal of Convex Analysis}, 11(2):251--266, 2004.

\bibitem{iusem2003convergence}
Alfredo~N Iusem.
\newblock On the convergence properties of the projected gradient method for
  convex optimization.
\newblock {\em Computational \& Applied Mathematics}, 22(1):37--52, 2003.

\bibitem{jaggi2013revisiting}
Martin Jaggi.
\newblock Revisiting {F}rank-{W}olfe: Projection-free sparse convex
  optimization.
\newblock In {\em ICML (1)}, pages 427--435, 2013.

\bibitem{krishnan2015barrier}
Rahul~G Krishnan, Simon Lacoste-Julien, and David Sontag.
\newblock Barrier frank-wolfe for marginal inference.
\newblock In {\em Advances in Neural Information Processing Systems}, pages
  532--540, 2015.

\bibitem{lacoste2016convergence}
Simon Lacoste-Julien.
\newblock Convergence rate of {F}rank-{W}olfe for non-convex objectives.
\newblock {\em arXiv preprint arXiv:1607.00345}, 2016.

\bibitem{lacoste2015global}
Simon Lacoste-Julien and Martin Jaggi.
\newblock On the global linear convergence of {F}rank-{W}olfe optimization
  variants.
\newblock In {\em Advances in Neural Information Processing Systems}, pages
  496--504, 2015.

\bibitem{nesterov2018lectures}
Yurii Nesterov.
\newblock {\em Lectures on convex optimization}, volume 137.
\newblock Springer, 2018.

\bibitem{nutini2017let}
Julie Nutini, Issam Laradji, and Mark Schmidt.
\newblock Let's make block coordinate descent go fast: Faster greedy rules,
  message-passing, active-set complexity, and superlinear convergence.
\newblock {\em arXiv preprint arXiv:1712.08859}, 2017.

\bibitem{nutini2019active}
Julie Nutini, Mark Schmidt, and Warren Hare.
\newblock "{A}ctive-set complexity" of proximal gradient: How long does it take
  to find the sparsity pattern?
\newblock {\em Optimization Letters}, 13(4):645--655, 2019.

\bibitem{pena2018polytope}
Javier Pe\~{n}a and Daniel Rodriguez.
\newblock Polytope conditioning and linear convergence of the {F}rank--{W}olfe
  algorithm.
\newblock {\em Mathematics of Operations Research}, 44(1):1--18, 2018.

\bibitem{sun2019we}
Yifan Sun, Halyun Jeong, Julie Nutini, and Mark Schmidt.
\newblock Are we there yet? manifold identification of gradient-related
  proximal methods.
\newblock In {\em The 22nd International Conference on Artificial Intelligence
  and Statistics}, pages 1110--1119, 2019.

\bibitem{wolfe1970convergence}
Philip Wolfe.
\newblock Convergence theory in nonlinear programming.
\newblock {\em Integer and nonlinear programming}, pages 1--36, 1970.

\bibitem{wright1993identifiable}
Stephen~J Wright.
\newblock Identifiable surfaces in constrained optimization.
\newblock {\em SIAM Journal on Control and Optimization}, 31(4):1063--1079,
  1993.

\end{thebibliography}
\end{document}